\documentclass[11pt,a4paper]{article}
\usepackage{amsmath,amsfonts,amsthm,amssymb,amsbsy,upref,color,graphicx,amscd,hyperref,makeidx,enumerate,bbm}

\usepackage[active]{srcltx}
\usepackage[latin1]{inputenc}

\DeclareMathOperator{\spt}{spt}

\def\V{\mathcal{V}}

\def\H{\mathcal{H}}
\def\vf{\varphi}
\def\g{\gamma}

\def\ds{\displaystyle}
\def\eps{{\varepsilon}}
\def\N{\mathbb{N}}
\def\O{\Omega}
\def\<{\prec}

\def\T{P}
\def\P{P}
\def\R{\mathbb{R}}

\def\HH{\mathcal{H}}

\def\M{\mathcal{M}}
\def\PP{\mathcal{P}}

\newcommand{\be}{\begin{equation}}
\newcommand{\ee}{\end{equation}}
\newcommand{\bib}[4]{\bibitem{#1}{\sc#2: }{\it#3. }{#4.}}

\newcommand{\cp}{\mathop{\rm cap}\nolimits}
\newcommand{\ind}{\mathbbm{1}}

\numberwithin{equation}{section}
\theoremstyle{plain}
\newtheorem{teo}{Theorem}[section]
\newtheorem{lemma}[teo]{Lemma}

\newtheorem{prop}[teo]{Proposition}
\newtheorem{deff}[teo]{Definition}

\theoremstyle{remark}
\newtheorem{oss}[teo]{Remark}

\theoremstyle{figure}

\title{Spectral optimization problems for potentials and measures}

\author{Dorin Bucur, Giuseppe Buttazzo, Bozhidar Velichkov}

\begin{document}
\maketitle

\begin{abstract}
In the present paper we consider spectral optimization problems involving the Schr\"odinger operator $-\Delta +\mu$ on $\R^d$, the prototype being the minimization of the $k$ the eigenvalue $\lambda_k(\mu)$. Here $\mu$ may be a capacitary measure with prescribed torsional rigidity (like in the Kohler-Jobin problem) or a classical nonnegative potential $V$ which satisfies the integral constraint $\ds \int V^{-p}dx \le m$ with $0<p<1$. We prove the existence of global solutions in $\R^d$ and that the optimal potentials or measures are equal to $+\infty$ outside a compact set.

\end{abstract}



%%%%%%%%%%%%%%%%%%%%%%%%%%%%%%
\section{Introduction}\label{sintro}

In shape optimization problems, as in all general optimization problems, proving the existence of a solution is a crucial step, which may reveal, in some cases, particularly difficult, due to the lack of compactness of minimizing sequences. In the case of shape optimization problems of spectral type, the existence issue was studied by Buttazzo and Dal Maso in \cite{budm93}, who proved that when the competing domains $\O$ are constrained to stay in a given {\it bounding box} $D\subset\R^d$, the optimization problem for a shape cost functional $F$
\be\label{shopt1}
\min\big\{F(\O)\ :\ \O\subset D,\ |\O|\le m\big\}
\ee
admits a solution, provided the assumptions below are satisfied:
\begin{itemize}
\item[i)]$F$ is lower semicontinuous with respect to the $\g$-convergence, that is
$$F(\O)\le\liminf_n F(\O_n)\qquad\hbox{whenever }\O_n\to_\gamma\O;$$
\item[ii)]$F$ is monotone decreasing with respect to the set inclusion, that is
$$F(\O_2)\le F(\O_1)\qquad\hbox{whenever }\O_1\subset\O_2.$$
\end{itemize}

Removing the bounded box constraint $\O\subset D$ in \eqref{shopt1} creates additional difficulties, and a general existence result, similar to the Buttazzo and Dal Maso one, is not available. The particular case of spectral optimization problems, in which
$$F(\O)=\Phi\big(\lambda(\O)\big),$$
being $\lambda(\O)=\big(\lambda_1(\O),\lambda_2(\O),\dots\big)$ the spectrum of the Dirichlet Laplacian in $\O$, made by the eigenvalues $\lambda_k(\O)$ of the operator $-\Delta$ on the space $H^1_0(\O)$, was considered in \cite{bulbk} and in \cite{pm10}. In these papers, by using two different approaches, an optimal solution for the problem
$$\min\big\{\Phi\big(\lambda_1(\O),\dots,\lambda_k(\O)\big)\ :\ \O\subset\R^d,\ |\O|\le m\big\}$$
is shown to exist, provided $\Phi$ is increasing in each variable and Lipschitz continuous. In addition, these solutions are proved to be bounded domains of $\R^d$ of finite perimeter. In particular, this applies to the case
$$\Phi(\lambda)=\lambda_k$$
which provides the optimal shape for the $k$-th eigenvalue $\lambda_k(\O)$ under the sole volume constraint.

The main purpose  of the paper is to consider optimization problems for Schr\"odinger potentials
\be\label{shopt2}
\min\big\{F(V)\ :\ V\ge0,\ V\in\V\big\},
\ee
in the spirit of \cite{bugeruve} (see the definition of the class $\V$ below). As in the case of shapes, when the competing potentials are assumed to be supported in a given {\it bounded} set $D\subset\R^d$, a general  result (see Theorem 4.1 of \cite{bugeruve}) provides the existence of an optimal potential under the assumptions:
\begin{itemize}
\item[i)]$F$ is lower semicontinuous with respect to the $\g$-convergence (see Section \ref{s25});
\item[ii)]$F$ is increasing, that is
$$F(V_1)\le F(V_2)\qquad\hbox{whenever }V_1\le V_2\hbox{ a.e.};$$
\item[iii)]the admissible set $\V$ is given by
$$\V=\left\{V\ge0,\ \int_D V^{-p}dx\le m\right\}\qquad\hbox{with $p>0$ and $m>0$ fixed}.$$
\end{itemize}
Removing the assumption $\spt V\subset D$, introduces several difficulties and a general result is not available. The first existence theorem  of this paper deals with this difficulty. Precisely, we prove an existence result for spectral problems of the form
$$\min\left\{\lambda_k(V)\ :\ V\ge0,\ \int_{\R^d}V^{-p}dx\le m\right\}$$
 when $0<p<1$, and moreover we prove that the optimal potential $V$ equals $+\infty$ outside a compact set. The techniques we use rely on new tools, as concentration-compactness results for capacitary measures (Section \ref{sconc}), on the concept of subsolutions for measure functionals (Section \ref{subsoil}) and a De Giorgi type argument (Lemma \ref{bndlemVh}).

 A second purpose of the paper is to study the minimization of the $k$-th eigenvalue under a torsion constraint, in the spirit of the Kohler-Jobin \cite{kojo} result for the first eigenvalue. 
  So in Section \ref{stors} we consider the spectral-torsion problem
$$\min\left\{\lambda_k(\mu)\ :\ \mu\in\M_{\cp}(\R^d),\ P(\mu)\le m\right\},$$
where $\M_{\cp}(\R^d)$ is the class of capacitary measures (see Section \ref{sssobolev}) and $P(\mu)$ is the torsion functional (see Section \ref{ssdirichlet}). Using similar techniques we
prove that optimal capacitary measures exist and that they are $+\infty$ outside a compact set. Nevertheless, we are not able to prove that the optimal measure is a domain, as in the particular cases $k=1,2$.  An interesting property we use in the proof, is concerned with the behavior of the heat equation solutions in unbounded sets:  as soon as the heat source is positive outside a compact set, the corresponding temperature has the same property.

%%%%%%%%%%%%%%%%%%%%%%%%%%%%%%
\section{Preliminaries}\label{sprel}

\subsection{Sobolev spaces and capacitary measures}\label{sssobolev}

We define the capacity of a generic set $E\subset\R^d$ as
$$\cp(E)=\inf\Big\{\|u\|^2_{H^1}:\ u\in H^1(\R^d),\ u=1\ \hbox{a.e. in a neighbourhood of }E\Big\},$$
where $\|u\|^2_{H^1}=\|u\|_{L^2}^2+\|\nabla u\|_{L^2}^2$. We say that a property $\PP(x)$ holds {\it quasi-everywhere}, if the set where $\PP(x)$ does not hold has zero capacity, i.e.
$$\cp\Big(\big\{x\in\R^d\ :\ \PP(x)\hbox{ does not hold}\big\}\Big)=0.$$

We say that a function $f:\R^d\to\R$ is {\it quasi-continuous}, if there is a sequence of open set $\omega_n\subset\R^d$ such that 
$$\cp(\omega_n)\to 0 \qquad \hbox{and} \qquad f:(\R^d\setminus\omega_n)\to\R\quad \hbox{is continuous}.$$
It is well-known (see for example \cite{hepi05, evgar, zie89}) that every Sobolev function $u\in H^1(\R^d)$ has a quasi-continuous representative $\widetilde u:\R^d\to\R$. Moreover, if $\widetilde u_1$ and $\widetilde u_2$ are two quasi-continuous representatives of the same class of equivalence $u\in H^1(\R^d)$, then $\widetilde u_1=\widetilde u_2$ quasi-everywhere. From now on we identify the Sobolev space $H^1(\R^d)$ with the space of quasi-continuous representatives
$$H^1(\R^d)=\Big\{u:\R^d\to\R:\ u\ \hbox{quasi-continuous},\ \int_{\R^d}\big(u^2+|\nabla u|^2\big)\,dx<+\infty\Big\},$$
and we note that each element of $H^1(\R^d)$ is a function defined up to a set of zero capacity. Moreover, we recall that if the sequence $u_n\in H^1(\R^d)$ converges in norm to $u\in H^1(\R^d)$, then $u_n$ converges quasi-everywhere to $u$.\\

We say that a regular Borel measure (possibly $+\infty$ valued) $\mu$ in $\R^d$ is a {\it capacitary measure}, if
$$\Big(\cp(E)=0\Big)\, \Rightarrow \, \Big(\mu(E)=0\Big),\qquad \hbox{for every}\ E\subset\R^d.$$

\begin{oss}
Any measure $\mu$ which is absolutely continuous with respect to the Lebesgue measure, is a capacitary measure. Indeed, if $\cp(E)=0$, then $|E|=0$ and so $\mu(E)=0$.
\end{oss}

Since the Sobolev functions are defined up to a set of zero capacity, the integral $\ \int_{\R^d}u^2\,d\mu\ $ is well defined, when $\mu$ is a capacitary measure, for every $u\in H^1(R^d)$. We say that the capacitary measures $\mu$ and $\nu$ are equivalent, if 
$$\int_{\R^d} u^2\,d\mu=\int_{\R^d} u^2\,d\nu,\ \hbox{for every }u\in H^1(\R^d).$$
We denote by $\M_{\cp}(\R^d)$ the space obtained as a quotient of the family of capacitary measures on $\R^d$ with respect to this equivalence relation. From now on we will identify a capacitary measure with its class of equivalence. On the space of capacitary measures $\M_{\cp}(\R^d)$, there is a partial order, induced by the testing with Sobolev functions, i.e. we say that $\mu\prec\nu$, if
$$\int_{\R^d}u^2\,d\mu\le\int_{\R^d}u^2\,d\nu,\ \hbox{for every}\ u\in H^1(\R^d).$$
We define the Sobolev space $H^1_\mu$ as 
$$H^1_\mu(\R^d)=\left\{u\in H^1(\R^d)\ :\ \int u^2\,d\mu<+\infty\right\}.$$
As it was proved in \cite{budm93}, the space $H^1_\mu$, equipped with the norm 
$$\|u\|_{H^1_\mu}^2=\|u\|_{H^1}^2+\|u\|_{L^2(\mu)}^2,$$
is a Hilbert space. By definition, we have that if $\mu\prec\nu$, then $H^1_\mu\supset H^1_\nu$.

Given $\mu,\nu\in\M_{\cp}(\R^d)$ we denote by $\mu\vee\nu$ the measure
$$\mu\vee\nu(E)=\sup\big\{\mu(A)+\nu(E\setminus A)\ :\ A\subset E\big\}.$$
It is easy to see that $\mu\vee\nu\in\M_{\cp}(\R^d)$ and that the corresponding Sobolev spaces verify:
$$H^1_{\mu\vee\nu}=H^1_\mu\cap H^1_\nu.$$

A typical example of a capacitary measure is the measure $I_\O $ associated to  a Borel set $\Omega$ 
\be\label{IO}
I_\O(E)=\begin{cases}0,\ \hbox{if}\ \cp(E\cap\O^c)=0,\\ +\infty,\ \hbox{if}\ \cp(E\cap\O^c)>0,\end{cases}
\ee 
for every $E\subset\R^d$. For a Borel set $\O\subset\R^d$, we use the notation $H^1_0(\O):=H^1_{I_\O}(\R^d)$. We note that (see \cite{bubu05, hepi05}) if $\O$ is an open set, then $H^1_0(\O)$ is the usual Sobolev space, obtained as the closure, with respect to the norm $\|\cdot\|_{H^1}$, of the smooth functions with compact support in $\O$, which we denote by $C^\infty_c(\O)$.

%%%%%%%%%%
\subsection{Torsional rigidity and torsion function}\label{ssdirichlet}

Let $\mu\in\M_{\cp}(\R^d)$ and fix $f\in L^p(\R^d)$ with $p\in[1,+\infty]$. For $u\in H^1_\mu\cap L^{p'}(\R^d)$, where $p'=\frac{p}{p-1}$, we define the functional
\be\label{functj}
J_{\mu,f}(u)=\frac12\int_{\R^d}|\nabla u|^2\,dx+\frac12\int_{\R^d}u^2\,d\mu-\int_{\R^d}fu\,dx
\ee
and the {\it torsional Energy} of $\mu$ as
\be\label{diren}
E(\mu)=\inf\big\{J_{\mu,1}(u):\ u\in H^1_\mu\cap L^1(\R^d)\big\}.
\ee
Since $J_{\mu,1}(0)=0$, we have that $E(\mu)\le0$. We call {\it torsion} of $\mu$ the nonnegative quantity $P(\mu):=-E(\mu)$, extending in this way the classical notion of torsional rigidity of a two dimensional simply connected domain up to multiplicative constant. We note that $P(\mu)$ can be $+\infty$, for example, in the case $\mu\equiv 0$. On the other hand, if $\mu=I_\O$, for some set $\O$ of finite Lebesgue measure, then $P(\mu)<+\infty$ and the functional $J_{\mu,1}$ has a unique minimizer in $H^1_0(\O)$.

We define the {\it torsion function} $w_\mu$ for a generic $\mu\in\M_{\cp}(\R^d)$ as
$$w_\mu=\sup_{R>0}w_R,$$
where $w_R$ is the unique minimizer of $J_{\mu\vee I_{B_R},1}$, i.e. the solution of
$$\min\left\{\frac12\int_{\R^d}|\nabla u|^2\,dx+\frac12\int_{\R^d}u^2\,d\mu-\int_{\R^d}u\,dx:\ u\in H^1_\mu\cap H^1_0(B_R)\right\}.$$
In the following we denote by $\M^{\P}_{\cp}(\R^d)$ the subclass of capacitary measures $\mu$ whose torsion $P(\mu)$ is finite.

The following result was proved in \cite{bubu12} and \cite{tesi} and relates the integrability of $w_\mu$ to the finiteness of the torsion $P(\mu)$ and to the compact embedding of $H^1_\mu$ into $L^1(\R^d)$.

\begin{teo}\label{bigth1234}
Let $\mu\in\M_{\cp}(\R^d)$ and let $w_\mu$ be its torsion function. Then the following conditions are equivalent:
\begin{enumerate}[(1)]
\item The inclusion $H^1_\mu\subset L^1 (\R^d)$ is continuous and there is a constant $C>0$ such that
\be\label{contmu}
\|u\|_{L^1}\le C\big(\|\nabla u\|_{L^2}^2+\|u\|_{L^2(\mu)}^2\big)^{1/2}\qquad\hbox{for every }u\in H^1_\mu.
\ee
\item The inclusion $H^1_\mu\subset L^1$ is compact and \eqref{contmu} holds.
\item The torsion  function $w_\mu$ is in $L^1(\R^d)$.
\item The torsion $P(\mu)$ is finite.
\end{enumerate}
Moreover, if the above conditions hold, then $w_\mu\in H^1_\mu\cap L^1(\R^d)$ is the unique minimizer of $J_{\mu,1}$ in $H^1_\mu$ and
$$C^2\le\int_{\R^d}w_\mu\,dx=2\P(\mu).$$
\end{teo}  
  
\begin{proof}
We first prove that {\it (3)} and {\it (4)} are equivalent. 

{\it(3)} $\Rightarrow$ {\it(4)}. Since the functions in $H^1_\mu\cap L^1$ with compact support are dense in $H^1_\mu\cap L^1$, we have
\begin{align}
\inf\Big\{ J_{\mu,1}(u): u\in H^1_\mu(\R^d)\cap L^1(\R^d)\Big\}&=\inf_{R>0}\inf\Big\{ J_{\mu,1}(u): u\in H^1_{\mu\vee I_{B_R}}(\R^d)\cap L^1(\R^d)\Big\}\nonumber\\
&=\inf_{R>0} J_{\mu,1}(w_R)=\inf_{R>0}\left\{-\frac12\int_{\R^d}w_R\,dx\right\}\label{bigpropwmu1}\\
&=-\frac12\int_{\R^d}w_\mu\,dx>-\infty,\nonumber
\end{align}
where the last equality is due to the fact that $w_R$ is increasing in $R$ and converges to $w_\mu$. Moreover, we have that $w_\mu\in H^1_\mu\cap L^1(\R^d)$ and $w_\mu$ minimizes $J_{\mu,1}$. Indeed, since $w_R$ converges to $w_\mu$ in $L^1(\R^d)$ and $w_R$ is uniformly bounded in $H^1_\mu$ by the inequality 
$$\int_{\R^d}|\nabla w_R|^2\,dx+\int_{\R^d}w_R^2\,dx\le 2\int_{\R^d}w_R\,dx\le 2\int_{\R^d}w_\mu\,dx,$$
we have that $w_\mu\in H^1_\mu$ and $J_{\mu,1}(w_\mu)\le\liminf_{R\to\infty}J_{\mu,1}(w_R).$

{\it(4)} $\Rightarrow$ {\it(3)}. By \eqref{bigpropwmu1}, we have that for every $R>0$, 
\begin{equation*}
\int_{\R^d}w_R\,dx\le-2\inf\left\{ J_{\mu,1}(u): u\in H^1_\mu(\R^d)\cap L^1(\R^d)\right\}<+\infty.
\end{equation*}
Taking the limit as $R\to\infty$, and taking in consideration again \eqref{bigpropwmu1}, we obtain
\be\label{bigpropwmu2}
\int_{\R^d}w_\mu\,dx=-2\inf\left\{ J_{\mu,1}(u): u\in H^1_\mu(\R^d)\cap L^1(\R^d)\right\}<+\infty.
\ee

Since the implication {\it (2)} $\Rightarrow$ {\it (1)} is clear, it is sufficient to prove that {\it(1)} $\Rightarrow$ {\it(4)} and {\it(4)} $\Rightarrow$ {\it(2)}. 

{\it(1)} $\Rightarrow$ {\it(4)}. Let $u_n\in H^1_\mu$ be a minimizing sequence for $J_{\mu,1}$ such that $u_n\ge 0$ and $J_{\mu,1}(u_n)\le 0$, for every $n\in\N$. Then we have
$$\frac12\int_{\R^d}|\nabla u_n|^2\,dx+\frac12\int_{\R^d}u_n^2\,d\mu\le \int_{\R^d}u_n\,dx\le C\Big(\|\nabla u_n\|_{L^2}^2+\|u_n\|_{L^2(\mu)}^2\Big)^{1/2},$$
and so $u_n$ is bounded in $H^1_\mu(\R^d)\cap L^1(\R^d)$. Suppose that $u$ is the weak limit of $u_n$ in $H^1_\mu$. Then 
$$\|u\|_{H^1_\mu}\le\liminf_{n\to\infty}\|u_n\|_{H^1_\mu},\qquad \int_{\R^d}u\,dx=\lim_{n\to\infty}\int_{\R^d}u_n\,dx,$$
where the last equality is due to the fact that the functional $\Big\{u\mapsto\int u\,dx\Big\}$ is continuous in $H^1_\mu$. Thus, $u\in H^1_\mu\cap L^1(\R^d)$ is the (unique, due to the strict convexity of $J_{\mu,1}$) minimizer of $J_{\mu,1}$ and so $E(\mu)=\inf J_{\mu,1}>-\infty$.

We now prove {\it (3)} $\Rightarrow$ {\it (1)}. Since, $w_\mu\in H^1_\mu\cap L^1(\R^d)$ is the minimizer of $J_{\mu,1}$ in $H^1_\mu\cap L^1(\R^d)$, we have that the following Euler-Lagrange equation holds:
\be\label{bigpropwmu3}
\int_{\R^d}\nabla w_\mu\cdot\nabla u\,dx+\int_{\R^d}w_\mu u\,d\mu=\int_{\R^d}u\,dx,\qquad\forall u\in H^1_\mu(\R^d)\cap L^1(\R^d).
\ee
Thus, for every $u\in H^1_\mu(\R^d)\cap L^1(\R^d)$, we obtain
\begin{align*}
\|u\|_{L^1}&\le\Big(\|\nabla w_\mu\|^2_{L^2}+\|w_\mu\|^2_{L^2(\mu)}\Big)^{1/2}\Big(\|\nabla u\|^2_{L^2}+\|u\|^2_{L^2(\mu)}\Big)^{1/2}\\
&=\|w_\mu\|^{1/2}_{L^1}\Big(\|\nabla u\|^2_{L^2}+\|u\|^2_{L^2(\mu)}\Big)^{1/2}.
\end{align*}
Since $H^1_\mu(\R^d)\cap L^1(\R^d)$ is dense in $H^1_\mu(\R^d)$, we obtain {\it (1)}.

{\it (3)} $\Rightarrow$ {\it (2)}. Following \cite[Theorem 3.2]{bubu12}, consider a sequence $u_n\in H^1_\mu$ weakly converging to zero in $H^1_\mu$ and suppose that $u_n\ge 0$, for every $n\in\N$. Since the injection $H^1(\R^d)\hookrightarrow L^1_{loc}(\R^d)$ is locally compact, we only have to prove that for every $\eps>0$ there is some $R>0$ such that $\int_{B_R^c}u_n\,dx\le\eps$. Consider the function $\eta_R(x):=\eta(x/R)$ where
$$\eta\in C^\infty_c(\R^d),\quad 0\le\eta\le 1,\quad \eta=1\ \hbox{on}\ B_1,\quad \eta=0\ \hbox{on}\ \R^d\setminus B_2.$$
Testing \eqref{bigpropwmu3} with $(1-\eta_R)u_n$, we have 
$$\int_{\R^d}\Big[u_n\nabla w_\mu\cdot\nabla(1-\eta_R)+(1-\eta_R)\nabla w_\mu\cdot\nabla u_n)\Big]\,dx+\int_{\R^d}w_\mu(1-\eta_R)u_n\,d\mu=\int_{\R^d}(1-\eta_R)u_n\,dx,$$
and using the identity $\|\nabla\eta_R\|_\infty=R^{-1}\|\nabla\eta\|_\infty$ and the Cauchy-Schwartz inequality, we have
$$\int_{B_{2R}^c}u_n\,dx\le R^{-1}\|u_n\|_{L^2}\|\nabla w_\mu\|_{L^2}+\|\nabla u_n\|_{L^2}\|\nabla w_n\|_{L^2(B_R^c)}+\|u_n\|_{L^2(\mu)}\left(\int_{B_R^c}w_\mu^2\,d\mu\right)^{1/2}, $$
which for $R$ large enough gives the desired $\eps$.
%Since $(4)\Rightarrow(3)$ is trivial, it remains to prove $(1)\Rightarrow(4)$.
\end{proof}

%%%%%%%%%%
\subsection{Infinity estimates}

\begin{lemma}\label{infttybndflem}
Let $\mu\in\M_{\cp}(\R^d)$ and consider a nonnegative function $f\in L^p(\R^d)$, where $p\in(d/2,+\infty]$. Suppose that $u\in H^1_\mu\cap L^{p'}(\R^d)$ minimizes $J_{\mu,f}$ in $H^1_\mu\cap L^{p'}(\R^d)$. Then we have for some constant $C$, depending on $d$ and $p$,
\be\label{infest}
\|(u-t)^+\|_{\infty}\le C\|f\|_{L^p}|\{u>t\}|^{2/d-1/p}\qquad\forall t\ge0.
\ee
\end{lemma}

\begin{proof}
We first notice that, being $u\in L^{p'}(\R^d)$, we have
$$|\{u>t\}|\le\frac{1}{t^{p'}}\int_{\R^d}u^{p'}dx<+\infty\qquad\forall t>0.$$
For every $t\in(0,\|u\|_\infty)$ and $\eps>0$, we consider the test function
$$u_{t,\eps}=u\wedge t +[t+(u-t-\eps)^+].$$
Since $u_{t,\eps}\le u$ and $J_{\mu,f}(u)\le J_{\mu,f}(u_{t,\eps})$, we get
$$\frac12\int_{\R^d}|\nabla u|^2\,dx-\int_{\R^d}fu\,dx\le \frac12\int_{\R^d}|\nabla u_{t,\eps}|^2\,dx-\int_{\R^d}fu_{t,\eps}\,dx,$$
and after some calculations
$$\frac12\int_{\{t< u\le t+\eps\}}|\nabla u|^2\,dx\le \int_{\R^d} f\left(u-u_{t,\eps}\right)\,dx\le\eps\int_{\{u>t\}}f\,dx.$$
By the co-area formula we have
%\be
%\int_{\O_{t+\eps}\setminus\O_t}|\nabla u|^2\,dx=\int_t^{t+\eps}ds\int_{\{u=s\}}|\nabla u|\,d\HH^{d-1},
%\ee
%and so 
%\be
%\lim_{\eps\to0}\frac{1}{\eps}\int_{\{t< u\le t+\eps\}}|\nabla w|^2\,dx=\int_{\{u=t\}}|\nabla u|\,d\HH^{d-1},
%\ee
%in all points $t\ge0$, which are Lebesgue points for the function $t\mapsto\int_{\{u=t\}}|\nabla u|\,d\HH^{d-1}$. Thus, by \eqref{inftybound130504e1}, we have that
$$\int_{\{u=t\}}|\nabla u|\,d\HH^{d-1}\le 2\int_{\{u>t\}}f\,dx\le 2\|f\|_{L^p}|\{u>t\}|^{1/p'}.$$
Setting $\vf(t)=|\{u>t\}|$, we have that
%\be
%\vf(t+\eps)-\vf(t)=-|\{t<u\le t+\eps\}|=-\int_t^{t+\eps}ds\int_{\{u=s\}}\frac{1}{|\nabla u|}\,d\HH^{d-1},
%\ee
%and so, if $t$ is a Lebesgue point for the function $t\mapsto\int_{\{u=s\}}|\nabla u|^{-1}\,d\HH^{d-1}$, we have that $\vf$ is differentiable in $t$ and 
\begin{align*}
\vf'(t)&=-\int_{\{u=t\}}\frac{1}{|\nabla u|}\,d\HH^{d-1}\\
&\le-\left(\int_{\{u=t\}}|\nabla u|\,d\HH^{d-1}\right)^{-1}P(\{u>t\})^2\\
&\le -\|f\|_{L^p}^{-1}\vf(t)^{-1+1/p}C_d\vf(t)^{2(d-1)/d}=-\|f\|_{L^p}^{-1}C_d\vf(t)^{(d-2)/d+1/p},
\end{align*}
where $C_d$ is the constant from the isoperimetric inequality in $\R^d$. Setting $\alpha=\frac{d-2}{d}+\frac1p$, we have that $\alpha<1$ and since the solution of the ODE
$$y'=-Cy^\alpha,\qquad y(t_0)=y_0,$$
where $t_0>0$, is given by
$$y(t)=\big(y_0^{1-\alpha}-(1-\alpha)C(t-t_0)\big)^{1/(1-\alpha)}.$$
Note that $\phi(t)\ge 0$, for every $t\ge0$, and $y(t)\ge \phi(t)$, if $\phi(t)>0$. Thus, we have that there is some $t_{\max}$ such that $\phi(t)=0$, for every $t\ge t_{\max}$. Taking $y_0=\phi(t_0)=|\{u>t_0\}|$, we have the estimate
\begin{align*}
\|(u-t_0)^+\|_{\infty}\le t_{\max}-t_0\le C\|f\|_{L^p}|\{u>t_0\}|^{2/d-1/p},
\end{align*}
for some constant $C$, which depends only on the dimension, if $d\ge 3$, and depends on $p$, if $d=2$.
\end{proof}

\begin{prop}\label{inftybndf11}
Let $\mu\in\M_{\cp}^{\T}(\R^d)$, $d\ge2$, $p\in(d/2,+\infty]$ and $f\in L^p(\R^d)$. Then there is a unique minimizer $u\in H^1_\mu$ of the functional $J_{\mu,f}: H^1_\mu\to \R$. Moreover, $u$ satisfies the inequality
\be\label{inftybndf21}
\|u\|_\infty\le C \P(\mu)^{\alpha}\|f\|_{L^p},
\ee
for some constants $C$ and $\alpha$, depending on the dimension $d$ and the exponent $p$. 
\end{prop}

\begin{proof}
We first note that for any $v\in H^1_\mu$ such that $J_{\mu,f}(v)\le 0$, we have
$$\int_{\R^d}|\nabla v|^2\,dx+\int_{\R^d}v^2\,dx\le2\int_{\R^d}fv\,dx\le 2\|f\|_{L^p}\|v\|_{L^{p'}}.$$
On the other hand $p>d/2$ implies $p'<\frac{d}{d-2}$ and so $p'\in[1,\frac{2d}{d-2}]$. Thus, using \eqref{contmu} with $C=P(\mu)^{1/2}$ and an interpolation, we obtain 
\be\label{inftybnffqqq}
\int_{\R^d}|\nabla v|^2\,dx+\int_{\R^d}v^2\,dx\le C_d P(\mu)^\alpha\|f\|^2_{L^p},
\ee
which in turn implies the existence of a minimizer $u$ of $J_{\mu,f}$, satisfying the same estimate. 

In order to prove \eqref{inftybndf21} it is sufficient to consider the case $f\ge 0$. In this case the solution is nonnegative $u\ge 0$ (since the minimizer is unique and $J_{\mu,f}(|u|)\le J_{\mu,f}(u)$) and, by Lemma \ref{infttybndflem}, we have that $u\in L^\infty$. We set $M:=\|u\|_\infty<+\infty$ and apply again Lemma \ref{infttybndflem} to obtain
$$\frac{M^2}{2}=\int_0^M(M-t)\,dt\le C\|f\|_{L^p}\int_0^M|\{u>t\}|^\beta\,dt\le C\|f\|_{L^p}M^{1-\beta}\|u\|_{L^1}^{\beta},$$
where we set $\beta=2/d-1/p\le 1$. Thus we obtain 
\be\label{inftybndf41}
M^{1+\beta}\le C\|f\|_{L^p}\|u\|_{L^1}^\beta,
\ee
and using \eqref{inftybnffqqq} with $v=u$, we get \eqref{inftybndf21}.
\end{proof}

\begin{prop}
Suppose that $\mu$ is a capacitary measure such that $w_\mu\in L^1(\R^d)$. Then $w_\mu\in L^\infty(\R^d)$ and vanishes at infinity:
$$\|w_\mu\|_\infty\le C_d\|w_{\mu}\|_{L^1}^{\frac{2}{d+2}} \qquad \hbox{and} \qquad \lim_{R\to\infty}\|w_\mu \ind_{B_R^c}\|_\infty=0,$$
where $C_d$ is a dimensional constant.
\end{prop}  
\begin{proof}
We set $w:=w_\mu$. Taking $f\equiv 1$ in Lemma \ref{infttybndflem}, we obtain
\be
\|(w-t)_+\|_\infty\le \|w_{\O_t}\|_\infty\le C_d|\{w>t\}|^{2/d}.
\ee
Thus, $w\in L^\infty(\R^d)$ and setting $M=\|w\|_\infty$, we have 
$$(M-t)^{d/2}\le  C_d^{d/2}|\{w>t\}|,$$
and integrating for $t\in(0,M]$, we get
$$\frac{2}{2+d}M^{\frac{d+2}{2}}=\int_0^M(M-t)^{d/2}\,dt\le C_d^{d/2}\int_0^M |\{w>t\}|\,dt=C_d^{d/2}\|w_\mu\|_{L^1},$$
which gives the first part of the claim. The second part was proved in \cite{bubu12}.
\end{proof}

\subsection{Schr\"odinger operators for capacitary measures}
Suppose that  $f\in L^p(\R^d)$ and $p\in[2,+\infty]$. For a capacitary measure $\mu\in\M_{\cp}^{\T}(\R^d)$, there is a unique minimizer $w_{\mu,f} \in H^1_\mu$ of the functional $J_{\mu,f}$, which also satisfies
\be\label{extended}
\int_{\R^d}\nabla w_{\mu,f}\cdot\nabla v\,dx+\int_{\R^d}w_{\mu,f}v\,d\mu=\int_{\R^d}fv\,dx,\qquad \hbox{for every}\ v\in H^1_\mu.
\ee

By definition, we say that $w_{\mu,f}$ solves the equation
$$-\Delta w_{\mu,f}+\mu w_{\mu,f}=f,\qquad w_{\mu,f}\in H^1_\mu.$$
Using $v=w_{\mu,f}$ as a test function in \eqref{extended}, we get that 
$$\int_{\R^d}|\nabla w_{\mu,f}|^2\,dx+\int_{\R^d}w_{\mu,f}^2\,d\mu=\int_{\R^d}w_{\mu,f}f\,dx\le \|f\|_{L^2}\|w_{\mu,f}\|_{L^2},$$
which in turn gives that there is a constant $C$ depending only on the dimension $d$ and the torsion $\P(\mu)$ such that
$$\|w_{\mu,f}\|_{L^2}\le C\|f\|_{L^2}.$$

%\begin{oss}
%If $p>d/2$ and $f\in L^p(\R^d)$, we can use Lemma \ref{infttybndflem} to prove that  that $w_{\mu,f}\in L^\infty(\R^d)$ and 
%$$\|w_{\mu,f}\|_\infty\le C\|f\|_{L^p},$$
%where $C$ is a constant depending only on $p$, $\|w_\mu\|_{L^1}$ and the dimension (see \cite{tesi}).
%\end{oss}

We call the resolvent of $\mu$, the continuous linear operator 
$$R_\mu:L^2(\R^d)\to L^2(\R^d),\qquad R_\mu(f)=w_{\mu,f}.$$
Since $\mu\in\M_{\cp}(\R^d)$, the operator $R_\mu$ is compact and so, it has a discrete spectrum $0\le \dots\le\Lambda_k(\mu)\le\dots\le\Lambda_2(\mu)\le\Lambda_1(\mu)$. Thus the spectrum of the unbounded Schr\"odinger operator $(-\Delta+\mu)$, associated to the bilinear form $Q_\mu(u):=\|\nabla u\|_{L^2}^2+\|u\|_{L^2(\mu)}^2$, is given by 
$$0<\lambda_1(\mu)\le \lambda_2(\mu)\le\dots\le\lambda_k(\mu)\le\dots,$$
where $\lambda_k(\mu)=\Lambda_k(\mu)^{-1}$. Moreover, we have the variational characterization 
$$\lambda_k(\mu)=\min_{S_k\subset H^1_\mu}\max_{u\in S_k}\frac{\int_{\R^d}|\nabla u|^2\,dx+\int_{\R^d}u^2\,d\mu}{\int_{\R^d}u^2\,dx},$$
where the minimum is over all $k$-dimensional subspaces $S_k$ of $H^1_\mu$. 

There is a sequence of eigenfunctions $u_k\in H^1_\mu$, orthonormal in $L^2(\R^d)$ and satisfying
$$-\Delta u_k+u_k\mu=\lambda_k(\mu)u_k,\qquad u_k\in H^1_\mu.$$
Moreover, $u_k\in L^\infty(\R^d)$ and we have the estimate (see \cite{davies} and \cite{tesi}) 
\be\label{inftyuk}
\|u_k\|_\infty\le e^{\frac{1}{8\pi}}\lambda_k(\mu)^{d/4}.
\ee

\begin{oss}[Scaling]\label{rescmu}
Let $\mu\in\M_{\cp}^{\T}(\R^d)$ be a capacitary measure of finite torsion  and let $u_k\in H^1_\mu$ be the $k$th eigenfunction of $(-\Delta+\mu)$.
Then we have 
$$-\Delta u_k+\mu u_k=\lambda_k(\mu)u_k,$$
and rescaling the eigenfunction $u_k$ with $t>0$, we have 
$$-\Delta \big(u_k(x/t)\big)+\mu_t u_k(x/t)=t^{-2}\lambda_k(\mu) u_k(x/t),\qquad u_k(\cdot/t)\in H^1_{\mu_t},$$
where the measure $\mu_t$ is defined as $\mu_t:=t^{d-2}\mu(\cdot/t)$, i.e. for every $\phi\in L^1(\mu)$, we have
\be\label{resc0mu}
\int_{\R^d}\phi(x/t)\,d\mu_t(x):=t^{d-2}\int_{\R^d}\phi\,d\mu.
\ee
Repeating the same argument for every eigenfunction, we have that
\be\label{resc1mu}
\lambda_k(\mu_t)=t^{-2}\lambda_k(\mu).
\ee
Analogously, for the energy function $w_\mu$ we obtain
$$-\Delta\big(w_\mu(x/t)\big)+t^{d-2}\mu(x/t)w_\mu(x/t)=t^{-2},\qquad w_\mu(\cdot/t)\in H^1_{\mu_t},$$
and, in particular, we have 
\be\label{resc2mu}
w_{\mu_t}(x)=t^2w_\mu(x/t) \qquad \hbox{and} \qquad E(\mu_t)=t^{d+2}E(\mu).
\ee
\end{oss}
   
\begin{oss}[Scaling of potentials]\label{rescV}
We note that if $\mu\in\M_{\cp}^{\T}(\R^d)$ is of the form $\mu=V(x)dx$, then $\mu_t=V_t(x)dx$, where $V_t(x):=t^{-2}V(x/t)$.  
\end{oss}

%%%%%%%%%%   
\subsection{The $\gamma$-distance on the space of capacitary measures}\label{s25}

We define the $\gamma$-distance between $\mu,\nu\in\M^{\P}_{\cp}(\R^d)$ as
$$d_\gamma(\mu,\nu)=\|w_\mu-w_\nu\|_{L^1},$$
where $w_\mu$ and $w_\nu$ are the torsion functions of $\mu$ and $\nu$, which are integrable by Theorem \ref{bigth1234}. In particular, we say that the sequence of capacitary measures $\mu_n\in\M^{\P}_{\cp}(\R^d)$ $\gamma$-converges to $\mu\in\M^{\P}_{\cp}(\R^d)$, if the sequence of energy functions $w_{\mu_n}$ converges in $L^1(\R^d)$ to the energy function $w_\mu$. It was first proved in \cite{dmmo86} and \cite{dmmo87} (see also \cite{dmgar} for a different approach) that if $\O$ is a bounded open set, then the space of capacitary measures in $\O$ 
$$\Big\{\mu\in\M^{\T}_{\cp}(\R^d):\  I_{\O} \prec \mu\Big\},$$
is compact and, in particular, complete with respect to the $\gamma$-distance. Using this result, it was proved in \cite{buc00} that the space $\M_{\cp}^{\T}(\R^d)$ endowed with the distance $d_\gamma$ is complete (we also refer to \cite{tesi} for a more direct approach).  

\begin{oss}\label{gammaimplies}
The $\gamma$-convergence implies the norm convergence of the resolvents $R_\mu$ and the $\Gamma$-convergence in $L^2(\R^d)$ of the norms $\|\cdot\|_{H^1_\mu}$. More precisely, we have: 
\begin{itemize}
\item If the sequence $\mu_n\in\mathcal{M}_{\cp}^{\T}(\R^d)$ $\gamma$-converges to $\mu\in\mathcal{M}_{\cp}^{\T}(\R^d)$, then the sequence of resolvents $R_{\mu_n}$ converges in norm to $R_\mu$, i.e. 
$$\lim_{n\to\infty}\|R_{\mu_n}-R_\mu\|_{\mathcal{L}(L^2(\R^d))}=0.$$
\item If the sequence $\mu_n\in\mathcal{M}_{\cp}^{\T}(\R^d)$ $\gamma$-converges to $\mu\in\mathcal{M}_{\cp}^{\T}(\R^d)$, then the sequence of functionals 
$$\|\cdot\|^2_{H^1_{\mu_n}}:L^2(\R^d)\to [0,+\infty],\qquad \|u\|^2_{H^1_{\mu_n}}=\begin{cases}\|\nabla u\|_{L^2}^2+\|u\|_{L^2(\mu_n)}^2+\|u\|_{L^2}^2,\ \hbox{if}\ u\in H^1_{\mu_n},\\ +\infty,\ \hbox{otherwise},\end{cases}$$
$\Gamma$-converges to $\|\cdot \|^2_{H^1_\mu}:L^2(\R^d)\to[0,+\infty]$, i.e. the following conditions are satisfied: 
\begin{enumerate}[($\Gamma$1)]
\item for every sequence $u_n\in L^2(\R^d)$, converging in $L^2(\R^d)$ to $u\in L^2(\R^d)$, we have 
$$\|u\|_{H^1_\mu}\le\liminf_{n\to\infty}\|u_n\|_{H^1_{\mu_n}};$$
\item for every $u\in H^1_\mu$, there is  sequence $u_n\in H^1_{\mu_n}$, converging to $u$ strongly in $L^2(\R^d)$ and such that
$$\|u\|_{H^1_\mu}=\lim_{n\to\infty}\|u_n\|_{H^1_{\mu_n}}.$$
\end{enumerate}
\end{itemize}
For a proof of these two facts we refer to \cite{tesi}.
\end{oss}

\begin{oss}\label{examRgam}
We note that the $\gamma$-convergence is not equivalent to the norm convergence of the resolvent operators $R_\mu\in\mathcal{L}(L^2(\R^d))$. In fact, one can construct a sequence of capacitary measures $\mu_n\in\mathcal{M}_{\cp}^{\T}(\R^d)$ such that 
\begin{equation}\label{examRgame1}
\|w_{\mu_n}\|_{L^1}=1 \qquad \hbox{and} \qquad \|R_{\mu_n}\|_{\mathcal{L}(L^2(\R^d))}\to 0.
\end{equation}
For example, let $\mu_n=I_{\Omega_n}$, where $\Omega_n$ is a disjoint union of $n$ balls $B_{n,k}:=B_{r_n}(x_k^n)$, $k=1,\dots,n$, of the same radius equal to $r_n>0$. Since $w_{\mu_n}=\sum_{k=1}^n w_{I_{B_{n,k}}}$, we can choose $r_n$ such that $\|w_{\mu_n}\|_{L^1}=1$. Since $r_n\to0$, as $n\to\infty$, we have that 
$$\|R_{\mu_n}\|_{\mathcal{L}(L^2(\R^d))}=\lambda_1^{-1}(\mu_n)=C_d r_n^2\to0,$$
which completes the construction of the sequence satisfying \eqref{examRgame1}. 
\end{oss}

\subsection{$(\Delta-\mu)$-harmonic functions}
In order to prove the boundedness of the local subsolutions for functionals of the form $E_f-E_1$, we will need the notion of $(\Delta-\mu)$-harmonic function.
\begin{deff}
Let $\mu\in\M_{\cp}^{\T}(\R^d)$ be a capacitary measure with finite torsion and let $B_R\subset\R^d$ be a given ball. For every $u\in H^1_\mu$ we will denote with $h_u$ the solution of the problem 
\be\label{harmu}
\min\Big\{\int_{B_r}|\nabla v|^2\,dx+\int_{B_R}v^2\,d\mu:\ v\in H^1_\mu,\ u-v\in H^1_0(B_R)\Big\}.
\ee
We will refer to $h_u$ as the $(\Delta-\mu)$-harmonic function on $B_R$ with boundary data $u$ on $\partial B_R$.
\end{deff}
 
The following Remark summarizes the main properties of the Harmonic functions, which we will use in the sequel.  
 
\begin{oss}\emph{Properties of the $(\Delta-\mu)$-harmonic functions.}

\begin{itemize}
\item \emph{(Uniqueness).} By the strict convexity of the functional in \eqref{harmu}, we have that the problem \eqref{harmu} has a unique minimizer, i.e. $h_u$ is uniquely determined;
\item \emph{(First variation).} Calculating the first variation of the functional from \eqref{harmu}, we have 
\be\label{harmu2}
\int_{\R^d}\nabla h_u\cdot\nabla\psi\,dx+\int_{\R^d}h_u\psi\,d\mu=0, \qquad \forall \psi\in H^1_\mu\cap H^1_0(B_R),
\ee
and conversely, if the function $h_u\in H^1_\mu$ satisfies \eqref{harmu2}, then it minimizes \eqref{harmu};
\item \emph{(Comparison principle).} If $u, w\in H^1_\mu$ are two functions such that $w\ge u$ on $\partial B_R$, then $h_u\le h_w$. Indeed, using $h_u\vee h_w\in H^1_\mu$ and $h_w\wedge h_u\in H^1_\mu$ to test the minimality of $h_w$ and $h_u$, respectively, we get $$\int_{\{h_u>h_w\}}|\nabla h_u|^2\,dx+\int_{\{h_u>h_w\}} h_u^2\,d\mu=\int_{\{h_u>h_w\}}|\nabla h_w|^2\,dx+\int_{\{h_u>h_w\}} h_w^2\,d\mu,$$
which implies that $h_w\wedge h_u$ is also minimizer of \eqref{harmu} and so $h_w\wedge h_u=h_u$.
\end{itemize}
\end{oss}

%%%%%%%%%%%%%%%%%%%%%%%%%%%%%% 
\section{Concentration-compactness principle for capacitary measures}\label{sconc}

In this section we introduce our main tools for studying the behaviur of minimizing sequences of functionals involving capacitary measures. Our main result is a concentration-compactness principle for capacitary measures, analogous to the concentration-compactness Theorem proved by Bucur in \cite{buc00}, which was the key argument in the proof of existence of optimal domains for $\lambda_k$ under measure constraint. Before we state the main Theorem, we need some preliminary results. We start by recalling a classical result due to P.L.Lions (see \cite{lions}).

\begin{teo}\label{concom}
Let $(f_n)_{n\in\N}$ be a sequence of nonnegative functions, uniformly bounded in $L^1(\R^d)$. Then, up to a subsequence, one of the following properties holds:
\begin{itemize}
\item \emph{Concentration.} There exists a sequence $(x_n)_{n\geq1}\subset \R^d$ with the property that for all $\epsilon>0$ there is some $R>0$ such that
$$\sup_{n\in\N}\int_{\R^d\setminus B_R(x_n)}f_n\,dx\le\epsilon.$$

\item \emph{Vanishing.} For each $R>0$
$$\lim_{n\to\infty}\left(\sup_{x\in\R^d}\int_{B_R(x)}f_{n}\,dx\right)=0.$$

\item \emph{Dichotomy.} For every $\alpha>1$, there is a sequence $x_n\in\R^d$ and an increasing sequence $R_n\to+\infty$ such that 
$$\lim_{n\to\infty}\int_{B_{\alpha R_n}(x_n)\setminus B_{R_n}(x_n)}f_n\,dx=0,$$
$$\liminf_{n\to0}\int_{B_{R_n}(x_n)}f_n\,dx>0\qquad\hbox{and}\qquad \liminf_{n\to0}\int_{\R^d\setminus B_{\alpha R_n}(x_n)}f_n\,dx>0.$$
\end{itemize}
\end{teo}

\begin{oss}
Since the inclusion $H^1(\R^d)\subset L^1_{loc}(\R^d)$ is compact, we have that if a sequence $(u_n)_{n\in\N}$ is bounded in $ L^1(\R^d)\cap H^1(\R^d)$ and has the concentration property, then there is a subsequence converging strongly in $L^1$.
\end{oss}

\begin{lemma}\label{uniftrunc}
Let $\mu\in\M_{\cp}^{\T}(\R^d)$. Then, for every $1<R_1<R_2\le +\infty$ we have
\be\label{uniftruncdouble}
d_\gamma\big(\mu,\mu\vee I_{B_{R_1}\cup B_{R_2}^c}\big)\le \int_{B_{2R_2}\setminus B_{R_1/2}}w_\mu\,dx+C(R_1^{-2}+R_2^{-2}),
\ee
where the constant $C$ depends only on the torsion $P(\mu)$ and the dimension $d$. 
\end{lemma}

\begin{proof}
We first consider the case $R_2=+\infty$. We set for simplicity $R=R_1$, $w_R=w_{\mu\vee I_{B_R}}$ and $\eta_R(x)=\eta(x/R)$, where 
$$\eta\in C^\infty_c(\R^d),\quad 0\le\eta\le 1,\quad \eta=1\ \hbox{on}\ B_1,\quad \eta=0\ \hbox{on}\ \R^d\setminus B_2.$$
Then we have 
\begin{align*}
d_\gamma(\mu,\mu\vee I_{B_{2R}})&=\int_{\R^d}(w_\mu-w_{2R})\,dx\\
&=2(J_\mu(w_{2R})-J_{\mu}(w_\mu))\le 2(J_\mu(\eta_R w_\mu)-J_\mu(w_\mu))\\
&=\int_{\R^d}|\nabla (\eta_R w_\mu)|^2\,dx+\int_{\R^d}\eta_R^2w^2\,d\mu-2\int_{\R^d}\eta_R w_\mu\,dx+\int_{\R^d}w_\mu\,dx\\
&=\int_{\R^d}w_\mu^2|\nabla\eta_R|^2+\nabla w_\mu\cdot\nabla(\eta_R^2w_\mu)\,dx+\int_{\R^d}\eta_R^2w^2\,d\mu\\
&\qquad-2\int_{\R^d}\eta_R w_\mu\,dx+\int_{\R^d}w_\mu\,dx\\
&=\int_{\R^d}w_\mu^2|\nabla\eta_R|^2\,dx+\int_{\R^d}\eta_R^2w\,dx-2\int_{\R^d}\eta_R w_\mu\,dx+\int_{\R^d}w_\mu\,dx\\
&=\int_{\R^d}w_\mu^2|\nabla\eta_R|^2\,dx+\int_{\R^d}(1-\eta_R)^2w_\mu\,dx\\
&\le \frac{\|\nabla\eta\|^2_\infty}{R^2}\|w_\mu\|_{L^2}+\int_{\R^d\setminus B_{R}}w_\mu\,dx,\\
\end{align*}
which concludes the proof. The case $R_2<+\infty$ is analogous.
\end{proof}

\begin{lemma}\label{dorinres}
Let $\mu,\mu'\in\M_{\cp}^{\T}(\R^d)$ be such that $\mu\prec\mu'$. Then, we have
$$\|R_\mu-R_{\mu'}\|_{\mathcal{L}(L^2)}\le C_{d,\mu} \left[d_{\gamma}(\mu,\mu')\right]^{1/d},$$
where $C_{d,\mu}$ is a constant depending only on the dimension $d$ and the torsion $P(\mu)$. 
\end{lemma}

\begin{proof}
The proof follows the same argument as in \cite[Lemma 3.6]{buc00} and we report it here for the sake of completeness. Let $f\in L^p$, $f\ge 0$, for some $p>d/2$. Then
\begin{align*}
\int_{\R^d}|R_\mu(f)-R_{\mu'}(f)|^p\,dx&\le\|R_\mu(f)-R_{\mu'}(f)\|^{p-1}_{\infty}\int_{\R^d}R_\mu(f)-R_{\mu'}(f)\,dx\\
&\le C^{p-1}\|f\|^{p-1}_{L^p}\int_{\R^d}f(w_{\mu}-w_{\mu'})\,dx\\
&\le C^{p-1}\|f\|_{L^p}^p\|w_\mu-w_{\mu'}\|_{L^{p'}},
\end{align*}
and so, $R_\mu-R_{\mu'}$ is a linear operator from $L^p$ to $L^p$ such that 
$$\|R_\mu-R_{\mu'}\|_{\mathcal{L}(L^p;L^p)}\le C^{p-1}\|w_\mu-w_{\mu'}\|_{L^{p'}}^{1/p},$$
where, by Proposition \ref{inftybndf11}, the constant $C$ depends on the norm $\|w_\mu\|_{L^1}$. Since $R_\mu-R_{\mu'}$ is a self-adjoint operator in $L^2$, we have that 
$$\|R_\mu-R_{\mu'}\|_{\mathcal{L}(L^{p'};L^{p'})}\le C^{p-1}\|w_\mu-w_{\mu'}\|_{L^{p'}}^{1/p},$$
and, finally, by interpolation 
$$\|R_\mu-R_{\mu'}\|_{\mathcal{L}(L^{2})}\le C^{p-1}\|w_\mu-w_{\mu'}\|_{L^{p'}}^{1/p}.$$
Now using the $L^\infty$ estimate on $w_\mu$, and taking $p=d$, we have the claim.
\end{proof}

\begin{teo}\label{ccmu}
Let $\mu_n\in\M_{\cp}^{\T}(\R^d)$ be a sequence of capacitary measures of uniformly bounded torsion $P(\mu_n)$. Then, up to a subsequence, one of the following situations occurs:
\begin{enumerate}[(i)]

\item[(i)](Compactness) There is a sequence $x_n\in\R^d$ such that  $\mu_n(x_n+\cdot)$ $\gamma$-converges.\\

\item[(ii)](Vanishing) The sequence of resolvents $R_{\mu_n}$ converges to zero in the operator norm of $\mathcal{L}(L^2(\R^d))$. Moreover, we have $\|w_{\mu_n}\|_{\infty}\to0$ and $\lambda_1(\mu_n)\to+\infty$, as $n\to\infty$.\\
 
\item[(iii)](Dichotomy) There are capacitary measures $\mu_n^1$ and $\mu_n^2$ such that:
\begin{itemize}
\item $dist(\O_{\mu_n^1},\O_{\mu_n^2})\to+\infty$, as $n\to\infty$;\\
\item $\mu_n\le\mu_n^1\wedge \mu_n^2$, for every $n\in\N$;\\
\item $d_\gamma(\mu_n,\mu_n^1\wedge \mu_n^2)\to 0$, as $n\to\infty$;\\
\item $\|R_{\mu_n}-R_{\mu_n^1\wedge\mu_n^2}\|_{\mathcal{L}(L^2)}\to0$, as $n\to\infty$.
\end{itemize}
\end{enumerate}
\end{teo}

\begin{proof}
Consider the sequence of corresponding energy functions $w_n:=w_{\mu_n}$. Since 
$$\|\nabla w_n\|_{L^2}^2+\|w_n\|_{L^2(\mu_n)}^2=\|w_{n}\|_{L^1}=2P(\mu_n),$$
we have that $w_{n}$ is bounded in $H^1(\R^d)\cap L^1(\R^d)$. We now apply the concentration compactness principle (Theorem \ref{concom}) to the sequence $w_n$. 

If the concentration (Theorem \ref{concom} {\it (1)}) occurs, then by the compactness of the embedding $H^1(\R^d)\subset L^1_{loc}(\R^d)$, up to a subsequence $w_n(\cdot+x_n)$ is concentrated in $L^1(\R^d)$ for some sequence $x_n\in\R^d$. If $x_n$ has a bounded subsequence, then $w_n$ converges (up to a subsequence) in $L^1(\R^d)$ and so, we have {\it (i1)}. If $|x_n|\to\infty$, by the same argument we obtain {\it (i2)}.\\

Suppose now that the \emph{vanishing} (Theorem \ref{concom} {\it (2)}) holds. We prove that {\it (ii)} holds. Since the sequence of norms $\|R_{\mu_n}\|_{\mathcal{L}(L^2)}$ is uniformly bounded, it is sufficient to prove that for every $\vf\in C^\infty_c(\O)$ the sequence $R_{\mu_n}(\vf)$ converges to zero strongly in $L^2(\R^d)$. Let $\vf\in C^\infty_c(\R^d)$ and let $\eps>0$. We choose $R>\eps^{-d}$ large enough and $N\in\N$ such that for every $n\ge N$, we have 
$$\int_{B_R} w_n\,dx\le\eps^d.$$ By Lemma \ref{uniftrunc} and Lemma \ref{dorinres}, we have that 
$$\|R_\mu(\vf)\|_{L^2}\le\|R_\mu(\vf)-R_{\mu\wedge I_{B_R^c}}(\vf)\|_{L^2}+\|R_{\mu\wedge I_{B_R^c}}(\vf)\|_{L^2}\le C\eps\|\vf\|_{L^2},$$
for some universal constant $C$. Thus we obtain the strong convergence in {\it (ii)}. 

We now prove that $\|w_{\mu_n}\|_\infty\to0$. Suppose by contradiction that there is $\delta>0$ and a sequence $x_n\in\R^d$ such that $w_{\mu_n}(x_n)>\delta$. Since $\Delta w_{\mu_n}+1\ge0$ on $\R^d$, we have that the function 
$$x\mapsto w_{\mu_n}(x)-\frac{r^2-|x-x_n|^2}{2d},$$
is subharmonic. Thus, choosing $r=\sqrt{d\delta}$, we have 
$$\int_{B_r(x_n)}w_{\mu_n}\,dx\ge w_{\mu_n}(x_n)-\frac{r^2}{2d}\ge \delta/2,$$
which contradicts Theorem \ref{concom} {\it (2)}.

Let $u_n\in H^1_{\mu_n}$ be the first eigenfunction for the operator $-\Delta+\mu_n$, normalized in $L^2(\R^d)$. By \eqref{inftyuk}, we have 
$$-\Delta u_n+\mu_nu_n=\lambda_1(\mu_n)u_n\le\lambda_1(\mu_n)\|u_n\|_\infty\le e^{1/(8\pi)}\lambda_1(\mu_n)^{(d+4)/4}.$$
Suppose that the sequence $\lambda_1(\mu_n)$ is bounded. Then by the weak maximum principle we have $u_{n}\le Cw_{\mu_n}$, for some constant $C$. Thus, we have
$$1=\int_{\R^d}u_n^2\,dx\le C^2\int_{\R^d}w_{\mu_n}^2\,dx\le C^2\|w_{\mu_n}\|_\infty\|w_{\mu_n}\|_{L^1}\to0,$$
which is a contradiction.

Suppose that the \emph{dichotomy} (Theorem \ref{concom} {\it (3)}) occurs. Choose $\alpha=8$ and let $x_n\in\R^d$ and $R_n\to\infty$ be as in Theorem \ref{concom} {\it (3)}. Then, setting
$$\mu_n^1=\mu_n\vee I_{B_{2R_n}(x_n)}\qquad\hbox{and}\qquad \mu_n^2=\mu_n\vee I_{B_{4R_n}(x_n)^c},$$
by Lemma \ref{uniftrunc} equation \eqref{uniftruncdouble}, we have
$$\lim_{n\to\infty}d_\gamma(\mu_n,\mu_n^1\wedge\mu_n^2)=0,$$
which, together with Lemma \ref{dorinres}, proves {\it (iii)}.
\end{proof}

%%%%%%%%%%%%%%%%%%%%%%%%%%%%%%
\section{Subsolutions of measure functionals}\label{subsoil}

Consider the functional
$$\mathcal{F}:\M_{\cp}^{\T}(\R^d)\to\R.$$
We say that the capacitary measure $\mu$ is a subsolution for $\mathcal{F}$, if we have
\be\label{ssmu}
\mathcal F(\mu)\le\mathcal F(\nu),\ \ \hbox{for every}\ \ \nu\in\M_{\cp}^{\T}(\R^d)\ \ \hbox{such that}\ \ \mu\<\nu.
\ee
We say that the capacitary measure $\mu$ is a local subsolution for $\mathcal{F}$, if 
\be\label{locssmu}
\exists\ \eps >0\ \ 
\mathcal F(\mu)\le\mathcal F(\nu),\ \ \hbox{for every}\ \ \nu\in\M_{\cp}^{\T}(\R^d)\ \ \hbox{such that}\ \ \mu\<\nu\ \ \hbox{and}\ \ d_\gamma(\mu,\nu)<\eps.
\ee
In the rest of this section we study the torsion function $w_\mu$ of a subsolution $\mu$ for a special class of functionals $\mathcal F$. Precisely, we consider spectral functionals with energy and mass penalization 
$$\mathcal{F}(\mu)=\lambda_k(\mu)+ P(\mu) \qquad \hbox{and} \qquad \mathcal{F}(\mu)=\lambda_k(\mu)+\int_{\R^d}\mu_{ac}^{-p}(x)\,dx,$$
where $p\in(0,1)$ and $\mu_{ac}$ denotes the absolute continuous part of $\mu$ with respect to the Lebesgue measure. Our main qualitative results, Theorem \ref{thsubV} and Theorem \ref{bndlbktor}, state that if $\mu$ is a subsolution, then it is constantly equal to infinity, outside a compact set. In other words, the sets of finiteness $\O_\mu:=\{w_\mu>0\}$ is bounded.

% 
%$M:\M_{\cp}^{\T}(\R^d)$ is a decreasing functional balancing the eigenvalue $\lambda_k(\mu)$.  is a spectral term of the form 
%$$\mathcal F_\lambda=\Phi(\lambda_1,\dots,\lambda_k)$$
%
%In our case $\mathcal{F}$ will consist of aprove that the subsolutions of a wide class of spectral functionals are bounded. involving the spectrum $\lambda_k(\mu)$ and 
%\begin{oss}
%In some cases we can use a smaller class of competitors in \eqref{ssmu}. If, for example, 
%$$\mathcal{F}(\mu)=E(\mu)+\int_{\R^d}\mu_{ac}^{-1}(x)\,dx,$$
%where $\mu_{ac}$ denotes the absolutely continuous part of $\mu$ with respect to the Lebesgue measure, then it is sufficient to consider absolutely continuous competitors $\nu=\nu_{ac}$.
%
%\end{oss} 

\subsection{From spectral to energy functionals}

In general, the spectral functionals are difficult to treat when the aim is to study the qualitative properties of the optimal measures. This difficulty is due to the fact that the eigenvalues $\lambda_k(\mu)$ are defined through a min-max principle, which makes any perturbation argument quite involved. In \cite{bulbk} was introduced a technique, which allows to concentrate only on energy functionals, at least when we aim to study the boundedness of the set of finiteness.

The following result is just a slight improvement of \cite[Lemma 3]{bulbk}, but is one of the crucial steps in the proof of existence of optimal measures for spectral-torsion functionals of the form $\mathcal F(\mu)=\lambda_k(\mu)+P(\mu)$.

\begin{lemma}\label{dorinlblemmu}
Let $\mu$ be a capacitary measure such that $w_\mu\in L^1(\R^d)$. Then for every capacitary measure $\mu\<\nu$ and every $k\in\N$, we have
\be\label{dorinlbleme0}
\Lambda_j(\mu)-\Lambda_j(\nu)\le k^2e^{\frac{1}{4\pi}}\lambda_k(\mu)^{\frac{d+4}2}\int_{\R^d}\big(R_\mu(w_\mu)w_\mu-R_\nu(w_\mu)w_\mu\big)\,dx.
\ee
\end{lemma}
\begin{proof}
Consider the orthonormal in $L^2(\R^d)$ family of eigenfunctions $u_1,\dots,u_k\in H^1_\mu$ corresponding to the compact self-adjoint operator $R_\mu:L^2(\R^d)\to L^2(\R^d)$. Let $P_k$ be the projection
$$P_k:L^2(\R^d)\to L^2(\R^d),\qquad P_k(u)=\sum_{j=1}^k \left(\int_{\R^d}uu_j\,dx\right) u_j.$$
Consider the linear space $V=Im(P_k)$, generated by $u_1,\dots,u_k$, and the operators $T_\mu$ and $T_\nu$ on $V$, defined by  
$$T_\mu=P_k\circ R_\mu\circ P_k\qquad\hbox{and}\qquad T_\nu=P_k\circ R_\nu\circ P_k.$$
It is immediate to check that $u_1,\dots,u_k$ and $\Lambda_1(\mu),\dots,\Lambda_1(\mu)$ are the eigenvectors and the corresponding eigenvalues of $T_\mu$. On the other hand, for the eigenvalues $\Lambda_1(T_\nu),\dots,\Lambda_k(T_\nu)$ of $T_\nu$, we have the inequality 
\be
\Lambda_j(T_\nu)\le\Lambda_j(\nu),\quad \forall j=1,\dots,k.
\ee
Indeed, by the min-max Theorem we have 
\begin{align*}
\Lambda_j(T_\nu)&=\min_{V_j\subset V}\ \max_{u\in V, u\perp V_j}\frac{\langle P_k\circ R_\nu\circ P_k (u),u\rangle_{L^2}}{\|u\|_{L^2}^2}\\
&=\min_{V_j\subset L^2}\ \max_{u\in V, u\perp V_j}\frac{\langle R_\nu (u),u\rangle_{L^2}}{\|u\|_{L^2}^2}\\
&\le\min_{V_j\subset L^2}\ \max_{u\in L^2, u\perp  V_j}\frac{\langle  R_\nu(u),u\rangle_{L^2}}{\|u\|_{L^2}^2}=\Lambda_j(\nu),
\end{align*}
where with $V_j$ we denotes a generic $(j-1)$-dimensional subspaces of $L^2(\R^d)$.

In conclusion, we obtain the estimate 
\begin{equation}\label{dorinlblemmue1}
0\le\Lambda_j(\mu)-\Lambda_j(\nu)\le \Lambda_j(T_\mu)-\Lambda_j(T_\nu)\le\|T_\mu-T_\nu\|_{\mathcal{L}(V)}.
\end{equation}
On the other hand, by definition of the norm $\|\cdot\|_{\mathcal{L}(V)}$, we have
\be\label{suptmutnu}
\begin{array}{ll}
\ds\|T_\mu-T_\nu\|_{\mathcal{L}(V)}&\ds=\sup_{u\in V}\frac{\langle (T_\mu-T_\nu)u,u\rangle_{L^2}}{\|u\|_{L^2}^2}=\sup_{u\in V}\frac{\langle (R_\mu-R_\nu)u,u\rangle_{L^2}}{\|u\|_{L^2}^2}\\
\\
&\ds\le\sup_{u\in V}\frac{1}{\|u\|_{L^2}^2}\int_{\R^d}\big(R_\mu(u)-R_\nu(u)\big)u\,dx.
\end{array}
\ee
Let $u\in V$ be the function for which the supremum in the r.h.s. of \eqref{suptmutnu} is achieved. We can suppose that $\|u\|_{L^2}=1$, i.e. that there are real numbers $\alpha_1,\dots,\alpha_k$, such that
$$u=\alpha_1 u_1+\dots+\alpha_ku_k\qquad\hbox{and}\qquad \alpha_1^2+\dots+\alpha_k^2=1.$$
Thus, we have
\be\label{dorinlblem0e0}
\begin{array}{ll}
\ds\|T_\mu-T_\nu\|_{\mathcal{L}(V)}&\ds\le\int_{\R^d}|R_\mu(u)-R_\nu(u)|\cdot |u|\,dx\\
&\ds\le\int_{\R^d}\Big|\sum_{j=1}^k \alpha_j\big(R_\mu(u_j)-R_\nu(u_j)\big)\Big|\cdot\Big(\sum_{j=1}^k|u_j|\Big)\,dx\\
&\ds\le\int_{\R^d}\Big(\sum_{j=1}^k \big|(R_\mu(u_j)-R_\nu(u_j)\big|\Big)\cdot\Big(\sum_{j=1}^k|u_j|\Big)\,dx\\
&\ds\le\int_{\R^d}\Big(\sum_{j=1}^k \big((R_\mu(|u_j|)-R_\nu(|u_j|)\big)\Big)\cdot\Big(\sum_{j=1}^k|u_j|\Big)\,dx,
\end{array}
\ee
where the last inequality is due to the linearity and the positivity of $R_\mu-R_\nu$.
We now recall that by \eqref{inftyuk}, we have $\|u_j\|_\infty\le e^{\frac{1}{8\pi}}\lambda_k(\mu)^{d/4}$, for each $j=1,\dots,k$. By the weak maximum principle applies for $u_j$ and $w_\mu$, we have
\be\label{inftyestokwmu}
|u_k|\le e^{\frac{1}{8\pi}}\lambda_k(\mu)^{\frac{d+4}{4}}w_\mu.
\ee
Using agains the positivity of $R_\mu-R_\nu$ and substituting \eqref{inftyestokwmu} in \eqref{dorinlblem0e0} we obtain the claim.  
\end{proof}

%\begin{lemma}\label{dorinlb}
%Let $\mu$ be a capacitary measure such that $w_\mu\in L^1(\R^d)$. Then for every capacitary measure $\nu\ge\mu$ and every $k\in\N$, we have
%%\be\label{dorinlbeq0}
%%\Lambda_j(\mu)-\Lambda_j(\nu)\le e^{1/4\pi}\lambda_k^{d/2}, 
%%\ee
%%
%%
%%
%\be
%\Lambda_j(\mu)-\Lambda_j(\nu)\le Cd_{\gamma}(\mu,\nu), 
%\ee
%for every $0<j\le k$, where $C$ is a constant depending only on $\lambda_k(\mu)$ and the dimension $d$.
%\end{lemma}
%\begin{proof}
%Reasoning as in Lemma \ref{dorinlblemmu}, by \eqref{dorinlblemmue1} and \eqref{dorinlblem0e0}, for each $j=1,\dots,k$, we have 
%\begin{align*}
%\Lambda_j(\mu)-\Lambda_j(\nu)&\le \int_{\R^d}\Big(\sum_{i=1}^k \big((R_\mu(|u_i|)-R_\nu(|u_i|)\big)\Big)\cdot\Big(\sum_{j=i}^k|u_i|\Big)\,dx\\
%&\le \Big(\sum_{j=i}^k\|u_i\|_\infty\Big)^2\int_{\R^d}(w_\mu-w_\nu)\,dx\\
%\end{align*}
%where $u_i\in H^1_\mu$ are the normalized eigenfunctions of $-\Delta+\mu$. Now the claim follows by \eqref{inftyuk}.
%\end{proof}

For a capacitary measure $\mu\in\M_{\cp}^{\T}(\R^d)$, we denote with $E_f(\mu)$ the Dirichlet Energy with respect to the function $f\in L^2$, i.e.
$$E_f(\mu)=\min_{u\in H^1_\mu} J_{\mu,f}=J_{\mu,f}(w_{\mu,f})=-\frac12\int_{\R^d}fw_{\mu,f},$$
which can be written in terms of the resolvent $R_\mu$ as
\be\label{reprEf}
E_f(\mu)=-\frac12\int_{\R^d}fR_\mu(f)\,dx.
\ee

\begin{prop}\label{mainsub}
Suppose that $\mathcal{G}:\M_{\cp}^{\T}(\R^d)\to\R$ is a given functional and that the capacitary measure $\mu\in\M_{\cp}^{\T}(\R^d)$ is a subsolution for the functional $\mathcal{F}=\lambda_k+\mathcal{G}$. Then there is a nonnegative function $f\in L^1(\R^d)\cap L^\infty(\R^d)$ vanishing at infinity, such that $\mu$ is a local subsolution for the functional $E_f(\mu)+\mathcal{G}(\mu)$.
\end{prop}
\begin{proof}
We first note that by Lemma \ref{dorinres}, we can choose $\eps>0$ such that, for every $\nu\in\mathcal{M}_{\cp}^{\T}(\R^d)$ with $\mu\<\nu$ and $d_{\gamma}(\mu,\nu)<\eps$, we have $\lambda_k(\mu)\le\lambda_k(\nu)\le 2\lambda_k(\mu)$. We now consider $\nu\in\M_{\cp}^{\T}(\R^d)$ with this property. Since $\mu$ is a subsolution for $\mathcal{F}$, we have
\begin{align*}
\mathcal{G}(\mu)-\mathcal{G}(\nu)&\le \lambda_k(\nu)-\lambda_k(\mu)\\
&\le 2\lambda_k(\mu)^2 \big(\Lambda_k(\mu)-\Lambda_k(\nu)\big)\\
&\le 4k^2e^{\frac{1}{4\pi}}\lambda_k(\mu)^{\frac{d+8}2}\big(E_{w_\mu}(\mu)-E_{w_\mu}(\nu)\big),
\end{align*}
where the last inequality is due to Lemma \ref{dorinlblemmu} and the representation \eqref{reprEf}. Note that for a generic $\rho\in L^2$ and $t>0$, we have $E_{t\rho}(\mu)=t^2E_\rho(\mu)$. Thus, setting $f=2ke^{\frac{1}{8\pi}}\lambda_k(\mu)^{\frac{d+8}4}w_\mu$, we have the claim.
\end{proof}

%
%\begin{oss}
%Suppose that the capacitary measure $\mu\in\M_{\cp}^{\T}(\R^d)$ is a subsolution for the functional $E_f(\mu)+\mathcal{G}(\mu)$, where $f\in L^\infty(\R^d)$. Then $\mu$ is also a subsolution for $\|f\|_\infty^2 E(\mu)+\mathcal{G}(\mu)$. 
%\end{oss}

\subsection{Subsolutions of spectral functionals with mass penalization}
In this subsection we prove that the subsolutions for the functionals of the form 
\be\label{lbk+mass}
\mathcal{F}(\mu)=\lambda_k(\mu)+\int_{\R^d}\mu_{ac}(x)^{-p}\,dx,
\ee
have bounded sets of finiteness, whenever $p\in(0,1)$. Our argument is based on Proposition \ref{mainsub} and the following Lemma, which is implicitly contained in \cite[Lemma 3.1]{deve}.

\begin{lemma}\label{uniftruncH}
Suppose that $\mu\in\M_{\cp}^{\T}(\R^d)$ is a capacitary measure of finite torsion. For the half-space $H=\{x\in\R^d:\ c+x\cdot \xi>0\}$, where the constant $c\in\R$ and the vector $\xi\in\R^d$ are given, we have
\be\label{uniftruncHe1}
d_\gamma(\mu,\mu\vee I_H)\le \sqrt{8\|w_\mu\|_\infty}\int_{\partial H}w_\mu\,d\HH^{d-1}-\int_{\R^d\setminus H}|\nabla w_\mu|^2\,dx-\int_{\R^d\setminus H}w_\mu^2\,d\mu+2\int_{\R^d\setminus H}w_\mu\,dx.
\ee 
\end{lemma}
\begin{proof}
For sake of simplicity, set $w:=w_\mu$, $M=\|w\|_{L^\infty}$, $c=0$ and $\xi=(0,\dots,0,-1)$. Consider the function
\be
v(x_1,\dots,x_d)=\begin{cases}
\begin{array}{ll}
M&,\ x_1\le -\sqrt M,\\
\frac12\Big(2M-(x_1+\sqrt {2M})^2\Big)&, -\sqrt {2M}\le x_1\le 0,\\
0&,\ 0\le x_1.
\end{array}
\end{cases}
\ee

Consider the function $w_H=w\wedge v\in H^1_0(H)\cap H^1_\mu$. 
\begin{align*}
d_\gamma(\mu,\mu\vee I_{H})&=\int_{\R^d}(w-w_{\mu\vee I_H})\,dx=2(J_\mu(w_{\mu\vee I_H})-J_{\mu}(w))\\
&\le 2(J_\mu(w_H)-J_\mu(w))\\
&\le\int_{\R^d}|\nabla (w_H)|^2-|\nabla w|^2\,dx-\int_{\R^d\setminus H}w^2\,d\mu+2\int_{\R^d}(w-w_H)\,dx\\
&\le\int_{\{-\sqrt{2M}<x_1\le 0\}}|\nabla (w_H)|^2-|\nabla w|^2\,dx-\int_{\R^d\setminus H}|\nabla w|^2\,dx\\
&\qquad\qquad\qquad\qquad-\int_{\R^d\setminus H}w^2\,d\mu+2\int_{\R^d}(w-w_H)\,dx\\
&\le 2\int_{\{-\sqrt{2M}<x_1\le 0\}}\nabla w_H\cdot\nabla (w_H-w)\,dx+2\int_{\{-\sqrt{2M}<x_1\le 0\}}(w-w_H)\,dx\\
&\qquad\qquad\qquad\qquad-\int_{\R^d\setminus H}|\nabla w|^2\,dx-\int_{\R^d\setminus H}w^2\,d\mu+2\int_{\R^d\setminus H}w\,dx\\
&= 2\int_{\{-\sqrt{2M}<x_1\le 0\}}\nabla v\cdot\nabla (w_H-w)\,dx+2\int_{\{-\sqrt{2M}<x_1\le 0\}}(w-w_H)\,dx\\
&\qquad\qquad\qquad\qquad-\int_{\R^d\setminus H}|\nabla w|^2\,dx-\int_{\R^d\setminus H}w^2\,d\mu+2\int_{\R^d\setminus H}w\,dx\\
&=\sqrt{8M}\int_{\partial H}w\,d\H^{d-1}-\int_{\R^d\setminus H}|\nabla w|^2\,dx-\int_{\R^d\setminus H}w^2\,d\mu+2\int_{\R^d\setminus H}w\,dx.
\end{align*}
\end{proof} 
 
 The next result has a double implication: on the one side it plays  a fundamental role in the proof of the existence of optimal potentials, and on the other side it gives a first qualitative result on them. The spirit of the proof follows a  classical argument introduced by De Giorgi, associated to an Alt-Caffarelli truncation argument \cite{altcaf}. 
 \begin{lemma}\label{bndlemVh}
Consider a nonnegative function $f\in L^\infty(\R^d)$ and a real number $p\in(0,1)$. Suppose that $\mu\in\M_{\cp}^{\T}(\R^d)$ is a local subsolution for the functional 
$$\mathcal{F}(\mu)=E_f(\mu)+\int_{\R^d}\mu_{ac}(x)^{-p}\,dx.$$
Then the set $\O_\mu=\{w_\mu>0\}$ is bounded.
\end{lemma}
\begin{proof}
We first recall that if $\nu\in\mathcal{M}_{\cp}^{\T}(\R^d)$ is such that $\mu\<\nu$, then the functional $R_\mu-R_\nu:L^2(\R^d)\to L^2(\R^d)$ is positive. Thus, we have  
$$E_f(\nu)-E_f(\mu)=\frac12\int_{\R^d}\big(fR_\mu(f)-fR_\nu(f)\big)\,dx\le\frac12\|f\|_\infty^2 \big(E(\nu)-E(\mu)\big),$$
and so, we can restrict our attention to the case $f\equiv1$.

For each $t\in\R$, we set 
\be
H_t=\{x\in\R^d:\ x_1=t\},\qquad H_t^+=\{x\in\R^d:\ x_1>t\},\qquad H_t^-=\{x\in\R^d:\ x_1<t\}.
\ee
We prove that there is some $t\in\R$ such that $|H_t^+\cap \O|=0$. For sake of simplicity, set 
$$w:=w_\mu, \qquad M=\|w\|_{\infty} \qquad \hbox{and} \qquad V(x)dx:=\mu_{ac}.$$ 
By Lemma \ref{uniftruncH} and the subminimality of $\O$, we have 

\be
\frac12\int_{H_t^+}|\nabla w|^2\,dx+\frac12\int_{H_t^+}w^2V\,dx+\int_{H_t^+}V^{-p}\,dx\le \sqrt{2M}\int_{H_t}w\,d\HH^{d-1}+\int_{H_t^+}w\,dx,
\ee
for every $t\in\R$. By aim to prove that the l.h.s. is grater than a power of $\int_{H_t^+}w\,dx$. Indeed, by the H\"older and Young inequalities, we have 
\begin{align*}
\int_{H_t^+}w^{\frac{2p}{p+1}}\,dx&\le\left(\int_{H_t^+}w^2V\,dx\right)^{\frac{p}{1+p}}\left(\int_{H_t^+}
V^{-p}\,dx\right)^{\frac1{1+p}}\\
&\le \frac{p}{1+p} \int_{H_t^+}w^2V\,dx+\frac1{1+p} \int_{H_t^+}V^{-p}\,dx.
\end{align*}
%
%On the other hand, by the Sobolev inequality, we have
%$$\left(\int_{H_t^+}w^{\frac{2d}{d-2}}\,dx\right)^{\frac{d-2}{d}}\le C_d\int_{H_t^+}|\nabla w|^2\,dx.$$
If $d \ge 3$, using the H\"older, Sobolev and Young inequalities we get

\begin{align*}
\left(\int_{H_t^+}w\,dx\right)^{\frac{d+2p}{d+1+p}}&\le \left(\int_{H_t^+}w^{\frac{2p}{p+1}}\,dx\right)^{\frac{(1+p)(d+2)}{2(d+1+p)}}\left(\int_{H_t^+}w^{\frac{2d}{d-2}}\,dx\right)^{\frac{(1-p)(d-2)}{2(d+1+p)}}\\
&\le\left(\int_{H_t^+}w^{\frac{2p}{p+1}}\,dx\right)^{\frac{(1+p)(d+2)}{2(d+1+p)}}\left(\int_{H_t^+}|\nabla  w|^2\,dx\right)^{\frac{(1-p)d}{2(d+1+p)}}\\
&\le\frac{(1+p)(d+2)}{2(d+1+p)}\int_{H_t^+}w^{\frac{2p}{p+1}}\,dx+\frac{(1-p)d}{2(d+1+p)}\int_{H_t^+}|\nabla  w|^2\,dx,
\end{align*}
which finally gives

\begin{align*}
\left(\int_{H_t^+}w\,dx\right)^{\alpha}&\le\frac{p(d+2)}{2(d+1+p)}\int_{H_t^+}w^2V\,dx+\frac{d+2}{2(d+1+p)}\int_{H_t^+}V^{-p}\,dx+\frac{(1-p)d}{2(d+1+p)}\int_{H_t^+}|\nabla  w|^2\,dx\\
&\le C\sqrt{2M}\int_{H_t}w\,d\HH^{d-1}+C\int_{H_t^+}w\,dx,
\end{align*}
where $\alpha=\frac{d+2p}{d+1+p}<1$ and $C$ is a constant depending on the dimension $d$ and the exponent $p$. Setting $$\phi(t):=\int_{H_t^+}w\,dx,$$
we have that
$$\phi'(t)=-\int_{H_t}w\,d\HH^{d-1},$$
and finally
$$\phi(t)^\beta\le -C\sqrt{2M}\phi'(t)+C\phi(t),$$
which gives that $\phi$ vanishes in a finite time. Repeating this argument in any direction we obtain that the support of $w$ is bounded.

If $d=2$, the same reasoning can be repeated replacing the Sobolev inequality by
$$\|u\|_{L^3(\R^2)} \le \frac32 \|u\|_{L^1(\R^2)}^\frac13 \|\nabla u\|_{L^2(\R^2)}^\frac23.$$
\end{proof}
 
\begin{teo}\label{thsubV}
Suppose that $\mu\in\M_{\cp}^{\T}(\R^d)$ is a subsolution for the functional $\mathcal{F}$ defined in \eqref{lbk+mass}. Then the set of finiteness $\O_\mu=\{w_\mu>0\}$ is bounded.
\end{teo}
\begin{proof}
By Proposition \ref{mainsub}, we have that $\mu$ is a local subsolution for a functional of the form $E_f(\mu)+\int_{\R^d}\mu_{ac}^{-\alpha}\,dx$. The conclusion follows by Lemma \ref{bndlemVh}.
\end{proof}

\subsection{Subsolutions for spectral-torsion functionals}

In this subsection we consider spectral functionals with torsion penalization of the form 
\be\label{lbk+tor}
\mathcal{F}(\mu)=\lambda_k(\mu)+P(\mu).
\ee 
We prove that any subsolution $\mu$ for $\mathcal{F}$ has a bounded set of finiteness $\Omega_\mu=\{w_\mu>0\}$. As in the case of functionals with mass penalization \eqref{lbk+mass} we will reduce our study to subsolutions of energy functionals. Our main instrument in proving the boundedness of $\Omega_\mu$ will be the following comparison principle "at infinity".

\begin{lemma}\label{compareinfty}
Consider a capacitary measure of finite torsion $\mu\in\M_{\cp}^{\T}(\R^d)$. Suppose that $u\in H^1_\mu$ is a solution of 
$$-\Delta u+\mu u=f,\qquad u\in H^1_\mu,$$
where $f\in L^1(\R^d)\cap L^\infty(\R^d)$ and $\lim_{x\to\infty}f(x)=0$. Then, there is some $R>0$, large enough, such that $u\le w_\mu$ on $\R^d\setminus B_R$.
\end{lemma}
\begin{proof}
Set $v=u-w_\mu$. We will prove that the set $\{v>0\}$ is bounded. Taking $v^+$ instead of $v$ and $\mu\vee I_{\{v>0\}}$ instead of $\mu$, we note that it is sufficient to restrict our attention to the case $v\ge 0$ on $\R^d$. We will prove the Lemma in four steps.\\

\emph{Step 1. There are constants $R_0>0$, $C_d>0$ and $\delta>0$ such that}
\begin{equation}\label{step1}
\left(\int_{\R^d}v^2\vf^{2(1+\delta)}\right)^{\frac{1}{1+\delta}}\le C_d \int_{\R^d}|\nabla\vf|^2 v^2\,dx,\qquad \forall \vf\in W^{1,\infty}_0(B_{R_0}^c).
\end{equation}

For any $\vf\in W^{1,\infty}(\R^d)$, we have that $v\vf^2\in H^1_\mu$ and so we may use it as a test function in 
\begin{equation*}
-\Delta v+\mu v=f-1,\qquad v\in H^1_\mu,
\end{equation*}
obtaining the identity 
\begin{equation}\label{compareinfty1}
\int_{\R^d}|\nabla(\vf v)|^2\,dx+\int_{\R^d}\vf^2v^2\,d\mu=\int_{\R^d}|\nabla\vf|^2 v^2\,dx+\int_{\R^d}v\vf^2(f-1)\,dx,\qquad \forall \vf\in W^{1,\infty}(\R^d).
\end{equation}
Let $R_0>0$ be large enough such that $1-f>\frac{4}{d+4}$. Then for any $\vf\in W^{1,\infty}_0(\R^d\setminus B_{R_0})$, we use the H\"older, Young and the Sobolev's inequalities together with \eqref{compareinfty1} to obtain 
\begin{equation}
\begin{array}{ll}
\ds\left(\int_{\R^d}v^2\vf^{\frac{2d+8}{d+2}}\,dx\right)^{\frac{d+2}{d+4}}&\ds \le \left(\int_{\R^d}(\vf v)^{\frac{2d}{d-2}}\,dx\right)^{\frac{d-2}{d+4}}\left(\int_{\R^d} v \vf^2\,dx\right)^{\frac{4}{d+4}}\\
\\
&\ds \le \frac{d}{d+4}\left(\int_{\R^d}(\vf v)^{\frac{2d}{d-2}}\,dx\right)^{\frac{d-2}{d}}+\frac4{d+4}\int_{\R^d} v \vf^2\,dx\\
\\
&\ds \le C_d\left(\int_{\R^d}|\nabla (\vf v)|^2\,dx+\int_{\R^d} v \vf^2(1-f)\,dx\right)\\
\\
&\ds \le C_d\int_{\R^d}|\nabla\vf|^2 v^2\,dx,
\end{array}
\end{equation}
where $C_d$ is a dimensional constant.\\
 
\emph{Step 2. There is some $R_1>0$ such that the function $M(r):=\frac{1}{d\omega_d r^{d-1}}\int_{\partial B_r}v^2\,d\HH^{d-1}$ is decreasing and convex on the interval $(R_1,+\infty)$.} 
 We first note that, for $R>0$ large enough, $\Delta v\ge (1-f)\chi_{\{v>0\}}\ge 0$ as an element of $H^{-1}(B_{R}^c)$. Since $\Delta (v^2)=2v\Delta v+2|\nabla v|^2$, we get that the function $U:=v^2$ is subharmonic on $\R^d\setminus B_R$. Now, the formal derivation of the mean $M$ gives 
 $$M'(r)=\frac{1}{d\omega_d r^{d-1}}\int_{\partial B_r}\nu\cdot\nabla U\,d\HH^{d-1},$$ 
 where $\nu_r$ is the external normal to $\partial B_r$. Let $R_1>0$ be such that $1\ge f$ on $\R^d\setminus B_{R_1}$. Then for any $R_1<r< R<+\infty$ we have 
\begin{align*}
d\omega_d\Big(R^{d-1}M'(R)-r^{d-1}M'(r)\Big)&=\int_{\partial B_{R}}\nu_R\cdot \nabla U\,d\HH^{d-1}-\int_{\partial B_{r}}\nu_r\cdot \nabla U\,d\HH^{d-1}\\
&=\int_{B_{R_2}\setminus B_{R_1}}\Delta U\,dx\ge 0.
\end{align*}
If we have that $M'(r)>0$ for some $r>R_1$, then $M'(R)>0$ for each $R>r$ and so $M$ is increasing on $[r,+\infty)$, which is a contradiction with the fact that $v$ (and so, $M$) vanishes at infinity. Thus, $M'(r)\le 0$, for all $r\in(R_1,+\infty)$ and so for every $R_1<r<R<+\infty$, we have 
$$R^{d-1}\big(M'(R)-M'(r)\big)\ge R^{d-1}M'(R)-r^{d-1}M'(r)\ge 0,$$
which proves that $M'(r)$ is also increasing.\\

\emph{Step 3. There are constants $R_2>0$, $C>0$ and $0<\delta<1/(d-1)$ such that the mean value function $M(r)$ satisfies the differential inequality}
\begin{equation}\label{step3}
M(r)\le C\big(r|M'(r)|+M(r)\big)^{\frac{d-1}{2}\delta}|M'(r)|^{1-\frac{d-2}{2}\delta},\qquad \forall r\in(R_2,+\infty).
\end{equation}

We first test the inequality \eqref{step1} with radial functions of the form $\vf(x)=\phi(|x|)$, where 
$$\phi(r)=0,\ \hbox{for}\ r\le R,\qquad \phi(r)=\frac{r-R}{\eps(R)},\ \hbox{for}\ R\le r\le R+\eps(R),\qquad \phi(r)=1,\ \hbox{for}\ r\ge R+\eps(R),$$
where $R>0$ is large enough and $\eps(R)>0$ is a given constant. As a consequence, we obtain 
\begin{equation}\label{compareinfty3}
\left(\int_{R+\eps(R)}^{+\infty}r^{d-1} M(r)\,dr\right)^{\frac{1}{1+\delta}}\le C_d \eps(R)^{-2}\int_{R}^{R+\eps(R)} r^{d-1} M(r)\,dr.
\end{equation}
By \emph{Step 2}, we have that for $R$ large enough:
\begin{itemize}
\item $M$ is monotone, i.e. $M(r)\le M(R)$ for $r\ge R$;
\item $M$ is convex $M(r)\ge M'(R)(r-R)+M(R)$ for $r\ge R$.
\end{itemize}
We now consider take $\eps(R)=\frac12\frac{M(R)}{|M'(R)|}$, i.e. $2\eps(R)$ is exactly the distance between $(R,0)$ and the intersection point of the $x$-axis with the line tangent to the graph of $M$ in $(R,M(R))$ (see Figure \ref{convex5}). With this choice of $\eps(R)$ we estimate both sides of \eqref{compareinfty3}, obtaining
\begin{equation}\label{compareinfty4}
\big(R+\eps(R)\big)^{\frac{d-1}{1+\delta}}\left(\frac14 M(R)\eps(R)\right)^{\frac{1}{1+\delta}}\le C_d \big(R+\eps(R)\big)^{d-1}\eps(R)^{-2}M(R),
\end{equation}
which, after substituting $\eps(R)$ with $\frac12\frac{M(R)}{|M'(R)|}$ gives \eqref{step3}.

\begin{figure}[t]
\begin{center}
\includegraphics[scale=0.5]{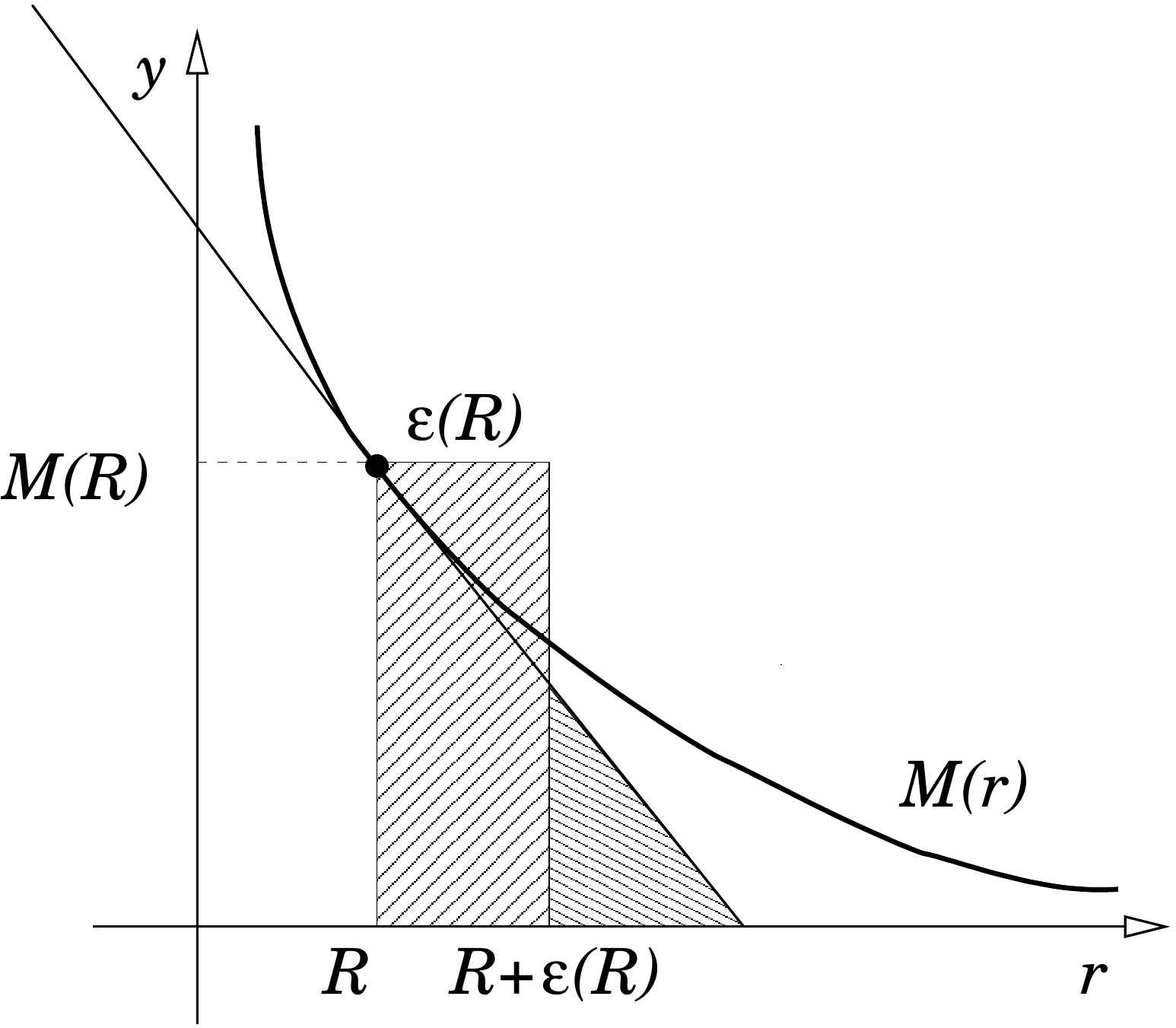}
\caption{We estimate the integral $\int_{R}^{R+\eps(R)}M(r)\,dr$ by the area of the rectangle on the right, while for the integral $\int_{R+\eps(R)}^{+\infty}M(r)\,dr$ is bounded from below by the area of the triangle on the right.}  
\label{convex5}
\end{center}
\end{figure}
\emph{Step 4. Each non-negative (differentiable a.e.) function $M(r)$, which vanishes at infinity and satisfies the inequality \eqref{step3} for some $\delta>0$ small enough, has compact support.}

Let $r\in(R_2,+\infty)$, where $R_2$ is as in \emph{Step 3}. We have two cases:
\begin{equation*}
\begin{array}{ll}
\ds (a)\ \hbox{If}\quad r|M'(r)|\ge M(r),\quad \hbox{then}\quad M(r)\le C_1 r^{\frac{(d-1)\delta}{2}}|M'(r)|^{1+\frac\delta2};\\
\\
\ds (b)\ \hbox{If}\quad r|M'(r)|\le M(r),\quad \hbox{then}\quad M(r)\le C_2 |M'(r)|^{1+\frac\delta2 \big(1-\frac{(d-1)\delta}2\big)}.
\end{array}
\end{equation*}
Choosing $\delta$ small enough, we get that in both cases $M$ satisfies the differential inequality
\begin{equation}\label{step43}
M(r)^{1-\delta_1}\le -C r^{\delta_2}M'(r),
\end{equation}
for appropriate constants $C>0$ and $0<\delta_1,\delta_2<1$. After integration, we have 
\begin{equation}\label{step44}
C'-C''r^{1-\delta_2}\ge M(r)^{\delta_1},
\end{equation}
for some constants $C', C''>0$, which concludes the proof.
\end{proof}

\begin{oss}
An alternative shorter proof of Lemma \ref{compareinfty} could be made by using viscosity solutions. For the sake of completeness we report this alternative proof in the appendix.
\end{oss}

\begin{lemma}\label{lembndlbktor}
Consider a capacitary measure of finite torsion $\mu\in\M_{\cp}^{\T}(\R^d)$. Let $f$ be a bounded measurable function converging to zero at infinity, i.e. $\lim_{R\to+\infty}\|f\|_{L^\infty(B_R^c)}=0$. If $\mu$ is a local subsolution for the functional $E_f(\mu)+P(\mu)$, then the set $\Omega_\mu=\{w_\mu>0\}$ is bounded.
\end{lemma}
\begin{proof} 
Let $\nu\in\mathcal{M}_{\cp}^{\T}(\R^d)$ be such that $\mu\<\nu$ and $d_\gamma(\mu,\nu)<\eps$. The subminimality of $\mu$ gives
$$E_{f}(\mu)-E(\mu)\le E_{f}(\nu)-E(\nu),$$
which can be stated in terms of $R_\mu$ and $R_\nu$ as 
\be\label{subminRom}
\int_{\R^d} \big( R_\mu(1)-fR_\mu(f)\big)\,dx\le \int_{\R^d} \big( R_\nu(1)-f R_\nu(f)\big)\,dx.
\ee
Moreover, by considering $f/2$ instead of $f$, we can suppose that the above inequality is strict whenever $w_\mu\neq w_\nu$.

We now show that choosing $\nu=\mu\vee I_{B_R}$, for some $R$ large enough, we can obtain equality in \eqref{subminRom}. Indeed, we have 
\begin{align*}
0&\ge \int_{\R^d}\big( R_\mu(1)-R_\nu(1)\big)-f\big(R_\mu(f)-R_\nu(f)\big)\,dx\\
&\ge \int_{\R^d}\big( R_\mu(1)-R_\nu(1)\big)-\big(R_\mu(\|f\|_\infty f)-R_\nu(\|f\|_\infty f)\big)\,dx\\
&= \int_{B_R}\big( R_\mu(1)-R_\nu(1)\big)-\big(R_\mu(\|f\|_\infty f)-R_\nu(\|f\|_\infty f)\big)\,dx\\
&\qquad+\int_{B_R^c}\big( R_\mu(1)-R_\mu(\|f\|_\infty f)\big)\,dx\\
&\ge \int_{B_R}\big( R_\mu(1)-R_\nu(1)\big)-\big(R_\mu(\|f\|_\infty f)-R_\nu(\|f\|_\infty f)\big)\,dx,
\end{align*}
where the last inequality holds for $R>0$ large enough and is due to Lemma \ref{compareinfty}. We now set for simplicity $w,u\in H^1_\mu$ to be respectively the solutions of 
$$-\Delta w+\mu w=1 \qquad \hbox{and} \qquad -\Delta u+\mu u=\|f\|_\infty f.$$
Thus, the functions 
$$h_w=R_\Omega(1)-R_\omega(1)\in H^1_\mu \qquad \hbox{and}\ \qquad h_u=R_\Omega(\|f\|_\infty f)-R_\omega(\|f\|_\infty f),$$
are $(\Delta-\mu)$-harmonic on the ball $B_R$. By the comparison principle, since $w\ge u$ on $\partial B_R$, we have that $h_w\ge h_u$ in $B_R$. Thus, for $R$ large enough and $\nu=\mu\vee I_{B_R}$, we have an equality in \eqref{subminRom}, which gives that $w_\mu=w_\nu$ and so $\Omega_\mu$ is bounded.
\end{proof}

\begin{teo}\label{bndlbktor}
Suppose that $\mu\in\M_{\cp}(\R^d)$ is a subsolution for the functional $\mathcal{F}$ from \eqref{lbk+tor}. Then the set of finiteness $\Omega_\mu=\{w_\mu>0\}$ is bounded.
\end{teo}
\begin{proof}
By Proposition \ref{mainsub}, we have that $\mu$ is a subsolution for a functional of the form $E_f(\mu)+P(\mu)$. By Lemma \ref{lembndlbktor} we conclude that $\Omega_\mu$ is bounded.
\end{proof}

%%%%%%%%%%%%%%%%%%%%%%%%%%%%%%
\section{Optimal potentials for Schr\"odinger operators}\label{spot}

In this subsection we consider optimization problems for spectral funcionals in $\R^d$. In particular, we consider the problem 
\begin{equation}\label{lbkpotrd}
\min\Big\{\lambda_k(V):\ V:\R^d\to[0,+\infty]\ \hbox{measurable},\ \int_{\R^d} V^{-p}\,dx=1\Big\},
\end{equation}
where $p\in(0,1)$. In the following proposition we prove that, under the integrability constraint in \eqref{lbkpotrd}, the spectrum of $-\Delta+V$ is discrete and thus, $\lambda_k(V)$ is well-defined.

\begin{prop}[Compactness of the embedding $H^1_V\hookrightarrow L^1$]\label{estw_Vpot}
Let $V:\R^d\to[0,+\infty]$ be a measurable function such that $\int_{\R^d}V^{-p}\,dx<+\infty$, where $p\in(0,1]$. Then the torsion function $w_V$, related to the measure $Vdx$, is integrable. In particular, the embedding $H^1_V\hookrightarrow L^1(\R^d)$ is compact and the spectrum of the operator $-\Delta+V$ is discrete. 
\end{prop}
\begin{proof}
See Example 3.10 of \cite{bubu12}.
\end{proof}

By Remark \ref{rescV}, the cost functional $\lambda_k(V)$ and the constraint $\int_{\R^d}V^{-p}\,dx$ have the following rescaling properties:
\begin{equation}\label{rescpot1}
\lambda_k(V_t)=t^{-2}\lambda_k(V)\qquad \hbox{and} \qquad \int_{\R^d}V_t^{-p}\,dx=t^{2p+d}\int_{\R^d}V^{-p}\,dx,
\end{equation}
where 
\begin{equation}\label{rescpot0}
V_t(x):=t^{-2}V(x/t).
\end{equation}
 
This rescaling property allows us to make the following remark.
 
\begin{oss}[Measure penalization]\label{rescpotosslag}
The potential $\widetilde V:\R^d\to[0,+\infty]$ is a solution of 
\begin{equation}\label{lbkpotrdlag}
\min\Big\{\lambda_k(V)+m\int_{\R^d} V^{-p}\,dx:\ V:\R^d\to[0,+\infty]\ \hbox{measurable}\Big\},
\end{equation}
if and only if, for every $t>0$, we have that $\widetilde V_t$, defined as in \eqref{rescpot0}, is a solution of 
\begin{equation}\label{lbkpotrdlag1}
\min\Big\{\lambda_k(V):\ V:\R^d\to[0,+\infty]\ \hbox{measurable}, \int_{\R^d} V^{-p}\,dx=\int_{\R^d} \widetilde V_t^{-p}\,dx\Big\},
\end{equation}
and the function 
$$f(t):=t^{-2}\lambda_k(\widetilde V)+mt^{2p+d}\int_{\R^d}\widetilde V^{-p}\,dx,$$
achieves its minimum, on the interval $(0,+\infty)$, in the point $t=1$.
\end{oss}

In the case $k=1$, the existence holds for every $p>0$. The following result was proved in \cite{bugeruve}.

\begin{prop}[Faber-Krahn inequality for potentials]\label{lbrdex}
For every $p>0$ there is a solution $V_p$ of the problem \eqref{lbkpotrd} with $k=1$. Moreover, there is an optimal potential $V_p$ given by  
\be\label{V-prd2}
V_p=\left(\int_{\R^d}|u_p|^{2p/(p+1)}\,dx\right)^{1/p}|u_p|^{-2/(1+p)},
\ee
where $u_p$ is a radially decreasing minimizer of 
\begin{align}\label{J-ard2}
&\min\Bigg\{\int_{\R^d}|\nabla u|^2\,dx+\left(\int_{\R^d}|u|^{2p/(p+1)}\,dx\right)^{(p+1)/p}\ :\ u\in H^1(\R^d),\ \int_{\R^d}u^2\,dx=1\Bigg\}.
\end{align}
Moreover, $u_p$ has a compact support, hence the set $\{V_p<+\infty\}$ is a ball of finite radius in $\R^d$.
\end{prop}

We now prove the existence of an optimal potential in the general case $k\ge 2$. 

\begin{teo}
Suppose that $p\in(0,1)$. Then, for every $k\in\N$, there is a solution of the problem \eqref{lbkpotrd}. Moreover, any solution $V$ of \eqref{lbkpotrd} is constantly equal to $+\infty$ outside a ball of finite radius.
\end{teo}
\begin{proof}
By Remark \ref{rescpotosslag}, every solution of \eqref{lbkpotrd} is a solution also of the penalized problem \eqref{lbkpotrdlag}, for some appropriately chosen Lagrange multiplier $m>0$. Thus, by Theorem \ref{thsubV} and Lemma \ref{bndlemVh}, we have that if $V$ is optimal for \eqref{lbkpotrdlag}, then it is constantly $+\infty$ outside a ball of finite radius. 

The proof of the existence part follows by induction on $k$. The first step $k=1$ being proved in Proposition \eqref{lbrdex}. We prove the claim for $k>1$, provided that the existence holds for all $1,\dots,k-1$. 

Let $V_n$ be a minimizing sequence for \eqref{lbkpotrd}. By Remark \ref{estw_Vpot}, we have that the sequence $w_{V_n}$ is uniformly bounded in $L^1(\R^d)$ and so, by Theorem \ref{ccmu}, we have two possibilities for the sequence of capacitary measures $V_ndx$: \emph{compactness} and \emph{dichotomy}.\\

If the compactness occurs, then there is a capacitary measure $\mu$ such that the sequence $V_ndx$ $\gamma$-converges to $\mu$. The sequence, $v_n:=V_n^{-p/2}$ is a bounded sequence in $L^{2}(\R^d)$ and so, up to a subsequence, we have that $v_n$ converges weakly in $L^{2}$ to some $v\in L^2(\R^d)$. We will prove that the function $V:=v^{-2/p}$ is a solution of \eqref{lbkpotrd}. The function $V$ satisfies the constraint from \eqref{lbkpotrd} and so it is sufficient to prove the inequality 
\begin{equation}\label{th1}
\lambda_k(V)\le \lambda_k(\mu)=\lim_{n\to\infty} \lambda_k(V_n),
\end{equation}
where the equality is just the continuity of $\lambda_k$ with respect to the $\gamma$-convergence. Since $V_ndx$ $\gamma$-converges to $\mu$, we have that the sequence of functionals $\|\cdot\|_{H^1_{V_n}}$ $\Gamma$-converges in $L^2(\R^d)$ to the functional $\|\cdot\|_{H^1_\mu}$ (see Remark \ref{gammaimplies}). In particular, for every $u\in H^1_\mu$, there is a sequence $u_n\in H^1_{V_n}$ which converges to $u$ in $L^2(\R^d)$ and is such that
\begin{align}
\int_{\R^d}|\nabla u|^2\,dx+\int_{\R^d} u^2\,d\mu &=\lim_{n\to\infty}\int_{\R^d}|\nabla u_n|^2\,dx+\int_{\R^d} u_n^2 V_n\,dx\nonumber\\
&=\lim_{n\to\infty}\int_{\R^d}|\nabla u_n|^2\,dx+\int_{\R^d} u_n^2 v_n^{-2/p}\,dx\label{ineqth2}\\
&\ge\int_{\R^d}|\nabla u|^2\,dx+\int_{\R^d} u^2 v^{-2/p}\,dx\nonumber\\
&=\int_{\R^d}|\nabla u|^2\,dx+\int_{\R^d} u^2V\,dx,\nonumber
\end{align}
where the inequality in \eqref{ineqth2} is due to strong-weak lower semicontinuity of integral functionals (see for instance \cite{busc}). Thus, for any $u\in H^1_\mu$, we have that 
$$\int_{\R^d} u^2\,d\mu\ge\int_{\R^d} u^2V\,dx,$$
and so, $V\<\mu$. Since $\lambda_k$ is an increasing functional, we obtain the first inequality in \eqref{th1} and so, $V$ is a solution of \eqref{lbkpotrd}.\\

If the dichotomy occurs, then we can suppose that $V_n=V_n^+\vee V_n^-$, where 
$$1/V_n=1/V_n^++1/V_n^-,\qquad \hbox{dist}\big(\{V_n^+<\infty\},\{V_n^-<\infty\}\big)\to+\infty.$$
Since $V_n$ is minimizing, there is $1\le l\le k-1$ such that
$$\lambda_k(V_n)=\lambda_{l}(V_n^+)\ge \lambda_{k-l}(V_n^-).$$
Taking the solutions, $V^+$ and $V^-$ respectively of 
\begin{equation*}
\min\Big\{\lambda_l(V):\ V:\R^d\to[0,+\infty]\ \hbox{measurable},\ \int_{\R^d} V^{-p}\,dx=\lim_{n\to\infty}\int_{\R^d}V_n^+\,dx\Big\},
\end{equation*}
\begin{equation*}
\min\Big\{\lambda_{k-l}(V):\ V:\R^d\to[0,+\infty]\ \hbox{measurable},\ \int_{\R^d} V^{-p}\,dx=\lim_{n\to\infty}\int_{\R^d}V_n^-\,dx\Big\},
\end{equation*}
in such a way that $\hbox{dist}\big(\{V^+<\infty\},\{V^-<\infty\}\big)>0$, we have that $V=V^+\wedge V^-$ is a solution of \eqref{lbkpotrd}.
\end{proof}

%%%%%%%%%%%%%%%%%%%%%%%%%%%%%%
\section{Optimal measures for spectral-torsion functionals}\label{stors}

In this section we consider consider the problem

\begin{equation}\label{lbkmurd}
\min\Big\{\lambda_k(\mu):\ \mu\in\mathcal{M}_{\cp}^{\T}(\R^d),\ \P(\mu)=c\Big\},
\end{equation}
where $c>0$ is a given constant. As in the case of potentials, we can substitute the constraint by a penalization.

\begin{oss}[Measure penalization]\label{rescmuosslag}
The capacitary measure $\widetilde\mu\in\mathcal{M}_{\cp}(\R^d)$ is a solution of 
\begin{equation}\label{lbkmurdlag}
\min\Big\{\lambda_k(\mu)+m P(\mu):\ \mu\in\mathcal{M}_{\cp}^{\P}(\R^d)\Big\},
\end{equation}
if and only if, for every $t>0$, the capacitary measure $\widetilde\mu_t$, defined as in Remark \ref{rescmu}, is a solution of 
\begin{equation}\label{lbkmurdlag1}
\min\Big\{\lambda_k(\mu):\ \mu\in\mathcal{M}_{\cp}^{\T}(\R^d),\ \P(\mu)=P(\widetilde\mu_t)\Big\},
\end{equation}
and the function 
$$f(t):=t^{-2}\lambda_k(\widetilde \mu)+mt^{2+d}\P(\widetilde\mu),$$
achieves its minimum, on the interval $(0,+\infty)$, for $t=1$.
\end{oss}

\begin{teo}\label{th62}
For every $k\in\N$ and $c<0$, there is a solution of the problem \eqref{lbkmurd}. Moreover, for any solution $\mu$ of \eqref{lbkmurd}, there is a ball $B_R$ such that $I_{B_R}\<\mu$.
\end{teo}
\begin{proof}
Suppose first that $\mu$ is a solution of \eqref{lbkmurd}. By Remark \ref{rescmuosslag}, $\mu$ is also a solution of the problem \eqref{lbkmurdlag}, for some constant $m>0$. In particular, $\mu$ is a subsolution for the functional
$$\mathcal{F}(\mu)=\lambda_k(\mu)+mP(\mu).$$ 
By Theorem \ref{bndlbktor}, we have that the set of finiteness $\Omega_\mu=\{w_\mu>0\}$ is bounded and so, there is a ball $B_R$ such that $I_{B_R}\<\mu$.

The proof of the existence part follows by induction on $k$. Suppose that $k=1$ and let $\mu_n$ be a minimizing sequence for the problem 
\begin{equation}\label{thsopmue1}
\min\Big\{\lambda_1(\mu)+mP(\mu):\ \mu\in\mathcal{M}_{\cp}^{\T}(\R^d)\Big\}.
\end{equation}
By the concentration-compactness principle (Theorem \ref{ccmu}), we have two possibilities: compactness and dichotomy. If the compactness occurs, we have that, up to a subsequence, $\mu_n$ $\gamma$-converges to some $\mu\in\mathcal{M}_{\cp}^{\T}(\R^d)$. Thus, by the continuity of $\lambda_1$ and $T$, we have that $\mu$ is a solution of \eqref{thsopmue1}. We now show that the dichotomy cannot occur. Indeed, if we suppose that $\mu_n=\mu_n^+\vee\mu_n^-$, where $\mu_n^+$ and $\mu_n^-$ have distant sets of finiteness $\Omega_{\mu_n^+}$ and $\Omega_{\mu_n^-}$, then 
$$\lambda_1(\mu_n)=\min\{\lambda_1(\mu_n^1),\lambda_1(\mu_n^+)\} \qquad \hbox{and} \qquad E(\mu_n)=E(\mu_n^+)+E(\mu_n^-).$$
Since, by Theorem \ref{ccmu}
$$\liminf_{n\to\infty}P(\mu_n^+)>0 \qquad\hbox{and}\qquad \liminf_{n\to\infty}P(\mu_n^-)>0,$$ 
we obtain that one of the sequences $\mu_n^+$ and $\mu_n^-$, say $\mu_n^+$ is such that 
$$\liminf_{n\to\infty}\big\{\lambda_1(\mu_n^+)+mP(\mu_n^+)\big\}<\liminf_{n\to\infty}
\big\{\lambda_1(\mu_n)+mP(\mu_n)\big\},$$
which is a contradiction and so, the compactness is the only possible case for $\mu_n$.

We now prove the claim for $k>1$, provided that the existence holds for all $1,\dots,k-1$. 

Let $\mu_n$ be a minimizing sequence for \eqref{lbkpotrd}. The sequence $w_{\mu_n}$ is uniformly bounded in $L^1(\R^d)$ and so, by Theorem \ref{ccmu}, we have two possibilities for the sequence of capacitary measures $\mu_n$: \emph{compactness} and \emph{dichotomy}.

If the compactness occurs, then there is a capacitary measure $\mu$ such that the sequence $\mu_n$ $\gamma$-converges to $\mu$, which by the continuity of $\lambda_k$ and the torsion $T$, is a solution of \eqref{lbkmurd}.

If the dichotomy occurs, then we can suppose that $\mu_n=\mu_n^+\vee \mu_n^-$, where the sets of finiteness $\Omega_{\mu_n^+}$ and $\Omega_{\mu_n^-}$ are such that
$$\hbox{dist}\big(\Omega_{\mu_n^+},\Omega_{\mu_n^-}\big)\to+\infty,\qquad P(\mu_n)=P(\mu_n^+)+P(\mu_n^-),$$
$$\lim_{n\to\infty}P(\mu_n^+)>0 \qquad \hbox{and} \qquad \lim_{n\to\infty}P(\mu_n^-)>0.$$
Since $\mu_n$ is a minimizing sequence, there is a constant $1\le l\le k-1$ such that
$$\lambda_k(\mu_n)=\lambda_{l}(\mu_n^+)\ge \lambda_{k-l}(\mu_n^-).$$
Taking the solutions, $\mu^+$ and $\mu^-$ respectively of 
\begin{equation*}
\min\Big\{\lambda_l(\mu):\ \mu\in\mathcal{M}_{\cp}(\R^d),\ P(\mu)=\lim_{n\to\infty}P(\mu_n^+)\Big\},
\end{equation*}
\begin{equation*}
\min\Big\{\lambda_{k-l}(\mu):\ \mu\in\mathcal{M}_{\cp}(\R^d),\ P(\mu)=\lim_{n\to\infty}P(\mu_n^-)\Big\},
\end{equation*}
in such a way that $\hbox{dist}\big(\Omega_{\mu^+},\Omega_{\mu^-}\big)>0$, we have that $\mu=\mu^+\vee \mu^-$ is a solution of \eqref{lbkmurd}.
\end{proof}

\begin{oss}
The Kohler-Jobin inequality (we refer to \cite{brasco} and the references therein for more details on this isoperimetric inequality) states that the ball $B$, such that $E(B)=c$, minimizes the first eigenvalue $\lambda_1(\Omega)$ under the constraint $E(\Omega)=c$, among all open sets $\Omega\subset\R^d$. Since the set $\{I_\Omega:\ \Omega\subset\R^d\ \hbox{open}\}\subset\mathcal{M}_{\cp}^{\T}(\R^d)$ is dense in $\mathcal{M}_{\cp}(\R^d)$ (see \cite{budm93}), we have that the measure $I_B$ solves \eqref{lbkmurd} for $k=1$. 
\end{oss}

{\bf Open Problem.} It would be interesting to establish whether the optimal measure $\mu$ given by Theorem \ref{th62} is actually a domain. Some numerical computations made by Beniamin Bogosel and Ioana Durus (private communication) seem to indicate that this is true and that, at least in dimension two, the optimal set is made by $k$ disjoint equal disks.

\appendix
\section{Appendix: an alternative proof of Lemma \ref{compareinfty}}

\begin{proof}[Proof of Lemma \ref{compareinfty}]
Set $v=u-w_\mu$. We will prove that the set $\{v>0\}$ is bounded. Taking $v^+$ instead of $v$ and $\mu\vee I_{\{v>0\}}$ instead of $\mu$, we note that it is sufficient to restrict our attention to the case $v\ge 0$ on $\R^d$. We now prove that if $v\in H^1(\R^d)$ is a nonnegative function such that
\begin{equation}\label{2ndproofe1}
-\Delta v+\mu v=f-1,\qquad v\in H^1_\mu,
\end{equation}
where $\mu\in\mathcal{M}_{\cp}^P(\R^d)$, $f\in L^\infty(\R^d)$ and $\lim_{|x|\to\infty}f(x)=0$, then $\{v>0\}$ is bounded. 

We first prove that there is some $R_0>0$ large enough such that the function $v$ satisfies the inequality $\Delta v\ge 1/2$ on $\R^d\setminus B_{R_0}$ in viscosity sense, i.e. for each $x\in\R^d\setminus B_{R_0}$ and each $\vf\in C^{\infty}(\R^d)$, satisfying $v\le \phi$ and $\vf(x)=v(x)$, we have that $\Delta\vf(x)\ge 1/2$.

Suppose that $\vf\in C^{\infty}(\R^d)$ is such that $v\le \phi$, $\vf(x)=v(x)$ and $\Delta\vf(x)<1/2-\eps$. By modifying $\vf$ and considering $\eps/2$ instead of $\eps$, we may suppose that, for $\delta>0$ small enough, $\{v+\delta>\vf\}\subset B_{R_0}^c$ and $\Delta\vf<1/2-\eps$ on the set $\{v+\delta>\vf\}$. Now taking $(v-\vf+\delta)^+\in H^1_\mu$ as a test function in \eqref{2ndproofe1}, we get that 
\begin{align*}
\int_{\R^d}(f-1)(v-\vf+\delta)^+\,dx&=\int_{\R^d}\nabla v\cdot\nabla(v-\vf+\delta)^+\,dx+\int_{\R^d}v(v-\vf+\delta)^+\,d\mu\\
&\ge \int_{\R^d}\nabla \vf\cdot\nabla(v-\vf+\delta)^+\,dx\\
&=-\int_{\R^d}(v-\vf+\delta)^+\Delta \vf \,dx\\
&>\Big(-\frac12+\eps\Big)\int_{\R^d}(v-\vf+\delta)^+\,dx,
\end{align*}
which gives a contradiction, once we choose $R_0>0$ large enough such that $f<1/4$ on $\R^d\setminus B_{R_0}$.

For $r\in (R_0,+\infty)$, we consider the function $M(r)=\sup_{\partial B_r}v$. Then $M:(R_0,+\infty)\to\R$ satisfies the inequality 
\begin{equation}\label{2ndproofe20}
M''(r)+\frac{d-1}{r}M'(r)\ge \frac12,\qquad\hbox{in viscosity sense}.
\end{equation} 
Indeed, let $r\in (R_0,+\infty)$ and $\phi\in C^\infty(\R)$ be such that $\phi(r)=M(r)$ and $\phi\ge M$. Then, taking a point $x_0\in\partial B_r$ such that $v(x)=M(r)$ (which exists due to the upper semi-continuity of $v$) and the function $\vf(x):=\phi(|x|)$, we have that $\vf\in C^\infty(\R^d)$, $\vf(x_0)=v(r)$ and $\vf\ge v$, which implies $\Delta\vf\ge1/2$ and so \eqref{2ndproofe20} holds.

%We now consider $R_1>R_0$ such that $M\le 1$ on $(R_1,+\infty)$. 
There is a constant $\eps_0>0$, depending on $R_0$, the dimension $d$ and $\|v\|_\infty$, such that the function $\phi\in C^\infty(\R)$, which solves 
\begin{equation}\label{2ndproofe30}
\phi''(r)+\frac{d-1}{r}\phi'(r)=\frac13,\qquad \phi(R_0)=\phi(R_0+\eps_0)=2\|v\|_\infty,
\end{equation}
changes sign on the interval $(R_0,R_0+\eps_0)$. We set 
$$t_0=\sup\big\{t: \{M\ge \phi+t\}\neq\emptyset \big\}>0.$$
Since $M$ is upper semi-continuous, there is some $r\in (R_0,R_0+\eps_0)$ such that $M(r)=\phi(r)+t_0$ and $M\le \phi+t_0$, which is a contradiction with \eqref{2ndproofe20}.
\end{proof}

%%%%%%%%%%%%%%%%%%%%%%%%%%%%%%

\bigskip
{\small\noindent
Dorin Bucur:
Laboratoire de Math\'ematiques (LAMA),
Universit\'e de Savoie\\
Campus Scientifique,
73376 Le-Bourget-Du-Lac - FRANCE\\
{\tt dorin.bucur@univ-savoie.fr}\\
{\tt http://www.lama.univ-savoie.fr/$\sim$bucur/}

\bigskip\noindent
Giuseppe Buttazzo:
Dipartimento di Matematica,
Universit\`a di Pisa\\
Largo B. Pontecorvo 5,
56127 Pisa - ITALY\\
{\tt buttazzo@dm.unipi.it}\\
{\tt http://www.dm.unipi.it/pages/buttazzo/}

\bigskip\noindent
Bozhidar Velichkov:
Scuola Normale Superiore di Pisa\\
Piazza dei Cavalieri 7, 56126 Pisa - ITALY\\
{\tt b.velichkov@sns.it}

\end{document}